\documentclass[11pt]{amsart}
\usepackage{amsfonts}
\usepackage{amsmath}
\usepackage{amssymb, amsthm}
\usepackage{mathtools}
\usepackage{float}
\usepackage{hyperref}
\usepackage{enumerate}
   \allowdisplaybreaks

\usepackage{comment}

\mathchardef\period=\mathcode`.
\DeclareMathSymbol{.}{\mathord}{letters}{"3B}
\usepackage{tikz-cd}
\usetikzlibrary{decorations.pathmorphing}
\newcommand*{\isoarrow}[1]{\arrow[#1,"\rotatebox{90}{\(\sim\)}"
]}
\makeatletter
\newtheorem*{rep@theorem}{\rep@title}
\newcommand{\newreptheorem}[2]{%
\newenvironment{rep#1}[1]{%
 \def\rep@title{#2 \ref{##1}}%
 \begin{rep@theorem}}%
 {\end{rep@theorem}}}
\makeatother
\tikzset{
  symbol/.style={
    draw=none,
    every to/.append style={
      edge node={node [sloped, allow upside down, auto=false]{$#1$}}}
  }
}
\newtheorem{lemma}{Lemma}[section]
\newreptheorem{lemma}{Lemma}
\newtheorem{theorem}[lemma]{Theorem}
\newreptheorem{theorem}{Theorem}
\newtheorem{corollary}[lemma]{Corollary}
\newreptheorem{corollary}{Corollary}
\newtheorem{prop}[lemma]{Proposition}
\newreptheorem{prop}{Proposition}

\theoremstyle{definition}
\newtheorem{defn}[lemma]{Definition}
\newtheorem{rem}[lemma]{Remark}

\newtheorem{question}[lemma]{Question}

\theoremstyle{remark}
\newtheorem*{rem*}{Remark}
\newtheorem*{note*}{Note}

\newcommand\isomto{\stackrel{\textstyle\sim}{\smash{\longrightarrow}\rule{0pt}{0.4ex}}}
\newcommand\restr[2]{{
  \left.\kern-\nulldelimiterspace 
  #1 
  \vphantom{\big|} 
  \right|_{#2} 
  }}
  \DeclareMathSymbol{.}{\mathord}{letters}{"3B}
\def\temp{&} \catcode`&=\active \let&=\temp
\begin{document}
\title{Remarks on the Bondal quiver}
\author{Benjamin Sung}
\address{Department of Mathematics, University of California, Santa Barbara, CA 93106, USA}
\email{bsung@ucsb.edu}
\date{}
\begin{abstract}
We study an admissible subcategory of the Bondal quiver which conjecturally does not admit any Bridgeland stability conditions. Specifically, we prove that its Serre functor coincides with the spherical twist associated with a $3$-spherical object. As a consequence, we obtain a classification of the spherical objects, deduce the non-existence of Serre-invariant stability conditions, and construct a natural spherical functor from its structure as a categorical resolution of the nodal cubic curve.
\end{abstract}
\maketitle
\tableofcontents
\section{Introduction}
A cornerstone of modern algebraic geometry is the study of algebraic varieties through their bounded derived categories, which often admit natural decompositions consistent with their Hodge structure. It has long been recognized that categories arising from such decompositions are not necessarily derived categories of a variety themselves, and it is critical for geometric applications to classify and to characterize their properties. A possible assumption to preclude pathologies is the existence of a Bridgeland stability condition, which is believed to exist on any smooth projective variety, and produces natural geometric structures from an abstract category, such as a moduli space of objects. This brings us to the guiding motivation for the present article, which is summarized by the following question:
\begin{question}\label{q:motivation}
Are there sharp criteria delineating triangulated categories admitting a Bridgeland stability condition from those that provably do not?
\end{question}
For example, the Jordan-H\"{o}lder property serves as a natural illustration of the principle in question~\ref{q:motivation}. A triangulated category $\mathcal{T}$ has the \textit{Jordan-H\"{o}lder property} if for any pair
\[
\mathcal{T} = \langle \mathcal{A}_1, \ldots, \mathcal{A}_n \rangle = \langle \mathcal{B}_1, \ldots, \mathcal{B}_m \rangle
\]
of semi-orthogonal decompositions, one has $m = n$, and there exists a permutation $\sigma \in S_m$ such that $\mathcal{B}_i \simeq \mathcal{A}_{\sigma(i)}$ for all $1 \leq i \leq n$. A number of counter-examples have been found, for example in \cite{kuznetsov2013simple,li2020refined}, which admit an admissible embedding into $D^b(X)$ with $X$ a smooth, projective variety. In particular, many such examples admit a pair of semi-orthogonal decompositions
\[
\mathcal{D} = \langle E_1, \ldots, E_k \rangle = \langle \mathcal{A}, F_1, \ldots, F_l \rangle
\]
where $E_i, F_j$ are exceptional objects, and $\mathcal{A}$ does not contain any exceptional objects. Moreover, $\mathcal{A}$ is generally quite pathological, in the sense that it contains a number of numerically trivial, spherical objects corresponding to the projections of the exceptional objects $F_j$, and it is expected that such categories do not admit a Bridgeland stability condition. 

On the other hand, a definitive answer even in concrete examples, is currently out of reach. Indeed, the determination of the existence of a stability condition requires an understanding of all hearts of a bounded t-structure on a triangulated category $\mathcal{D}$, and thus, non-existence is only known in rather trivial cases. For example, a (quasi)-phantom category $\mathcal{D}$, satisfies $K_0(\mathcal{D}) \otimes \mathbb{Q} = 0$ by definition, and hence there does not exist any numerical Bridgeland stability condition since the central charge cannot map any element to $0 \in \mathbb{C}$.

In an effort to study general properties of categories not expected to admit a stability condition, we will perform a systematic analysis of an admissible subcategory $P^\perp \subset D^b(Q)$, with $Q$ the Bondal quiver defined by the following quiver with relations
\[
Q = \Bigg(
\begin{tikzcd}
\bullet \arrow[r,bend left,"\alpha_1"] \arrow[r,bend right,swap,"\beta_1"] & \bullet   \arrow[r,bend left,"\alpha_2"] \arrow[r,bend right,swap,"\beta_2"] & \bullet
\end{tikzcd}\Bigg\vert \; \beta_2 \circ \alpha_1 = \alpha_2 \circ \beta_1 = 0 \Bigg), \quad
\]
and $D^b(Q)$ admitting the decompositions
\[
D^b(Q) = \langle P_1, P_2 ,P_3 \rangle = \langle P^\perp, P \rangle, \quad \begin{tikzcd}
P: \mathbb{C} \arrow[r, shift left, "1"] \arrow[r, shift right,"0", labels = below]& \mathbb{C}\arrow[r, shift left, "1"] \arrow[r, shift right,"0", labels = below]& \mathbb{C}
\end{tikzcd}
\]
with $P_i$ projective modules, $P$ an exceptional object, and $P^\perp$ not containing any exceptional objects. Moreover, this implies that $D^b(Q)$ does not satisfy the Jordan-H\"{o}lder property~\cite{kuznetsov2013simple}. Such an example lies at the nexus of various interesting phenomena in the theory of derived categories. On the one hand, $P^\perp$ is the minimal categorical resolution of the nodal cubic curve~\cite{2009arXiv0905.1231B}, but on the other hand, $P^\perp$ also contains a $3$-spherical object which is trivial in the Grothendieck group. In particular, $P^\perp$ exhibits many of the phenomena expected to hold for categorical resolutions of nodal singularities, but also exhibits $3$-dimensional phenomena and hence its study is not amenable to many techniques available for studying derived categories of curves and surfaces.

Specifically, we will systematically study the following phenomena for the category $P^\perp$.
\begin{enumerate}
\item
Serre functors on noncommutative schemes
\item
New spherical functors and derived symmetries
\item
Non-existence of Serre-invariant stability conditions
\end{enumerate}

\subsection{Summary of results}
We now discuss the main results of this article in more detail. We first prove the following theorem, describing the Serre functors $S$ and $S_{P^\perp}$ on $D^b(Q)$ and $P^\perp$ respectively. In the following $S(P)$ denotes the Serre dual of the exceptional object $P$ and $\mathbb{L}_{S(P)}$ denotes the left mutation with respect to this object. The object $E$, which will be described at length in Lemma~\ref{lem:E}, is constructed as an extension of $S(P)$ and $P$ and is a $3$-spherical object with spherical twist $\mathbb{T}_E$.
\begin{reptheorem}{thm:serremutation}
There exists a natural isomorphism of functors 
\[S^{-1}[1]|_{P^\perp} \simeq \mathbb{L}_{S(P)} \colon P^\perp \rightarrow \prescript{\perp}{}{P}
\]\end{reptheorem}
\begin{repcorollary}{cor:serrepperp}
The Serre functor on the category $P^\perp$ satisfies the following: $S^{-1}_{P^\perp} \simeq \mathbb{T}_E[-1]$.
\end{repcorollary}
As a consequence of our computation of the Serre functor, we obtain a classification of spherical objects in $P^\perp$.
\begin{repcorollary}{cor:unique}
The object $E$ is the unique $3$-spherical object in the category $P^\perp$.
\end{repcorollary}
Moreover, we may easily deduce the following, where we recall that a Bridgeland stability condition $\sigma$ is \textit{Serre-invariant} if  $S \cdot \sigma =\sigma \cdot g$ for some $g \in \widetilde{GL}_2^+(\mathbb{R})$. 
\begin{repcorollary}{cor:serreinv}
The category $P^\perp \subset D^b(Q)$ with $Q$ the Bondal quiver does not admit any Serre-invariant stability conditions. 
\end{repcorollary}
Applying corollary~\ref{cor:serrepperp}, we then turn to the composition of functors induced from the categorical resolution of the nodal cubic curve $X_0 \subset \mathbb{P}^2$.
\[
P^\perp \rightarrow Perf(X_0) \rightarrow D^b(\mathbb{P}^2)
\]
Our main goal will be to demonstrate that this is indeed a spherical functor, which opens the possibility of constructing new derived symmetries arising from categorical resolutions of singularities.


\subsection{Related works}
In \cite{kuznetsov2021serre}, the authors developed a general approach to compute the Serre functor for residual categories of Fano complete intersections. Though the setup is different from ours, we hope that Proposition~\ref{prop:bondalsphericalfunctor} could be useful in identifying with their setting.

In~\cite{li2020refined}, the authors studied an admissible subcategory of an Enriques surface. This category is quite similar to $P^\perp$ in many ways; it also admits a number of $3$-spherical objects which have trivial class in the Grothendieck group and which arise similarly as $E$ did from an ambient category. We note that our results parallel their study of Enriques categories in many respects. 

Finally, the authors in~\cite{2209.12853} studied categorical resolutions of nodal singularities in dimension $\geq 2$. We note that our results in section~\ref{sec:spherical} parallel theirs in a number of ways; for example, we demonstrate that the kernel of our categorical resolution is also generated by a unique spherical object.
\subsection{Organization}
The organization of this paper is as follows. In Section~\ref{sec:bondal}, review and prove several fundamental facts on the category $D^b(Q)$ associated to the Bondal quiver. In Section~\ref{sec:serre}, we develop a general approach and conclude with an explicit computation for the Serre functor on $D^b(Q)$ and $P^\perp$ in theorem~\ref{thm:serremutation}. In Section~\ref{sec:spherical}, we construct a natural spherical functor from $P^\perp$ to $D^b(\mathbb{P}^2)$, and prove Proposition~\ref{prop:bondalsphericalfunctor}.
\subsection*{Acknowledgements.}
I would like to thank my advisor, Emanuele Macr\`i, for extensive discussions throughout the years. I am grateful to to the University of Paris-Saclay for hospitality during the completion of this work. This work was partially supported by the NSF Graduate Research Fellowship under grant DGE-1451070 and by the ERC Synergy Grant ERC-2020-SyG-854361-HyperK.

\section{The Bondal Quiver}\label{sec:bondal}
In this section, we will introduce several basic facts associated with the Bondal quiver. We first introduce the basic building blocks of the bounded derived category of representations of the quiver, namely the projective and injective modules and some of their properties. Lemmas~\ref{lem:p} and \ref{lem:E} are of particular importance, and together with Corollary~\ref{cor:nonequiv}, these results highlight the particularly interesting aspects of this category. 

Fix $k$ an algebraically closed field of characteristic $0$ and let $Q$ be the Bondal quiver given by the following:
\begin{equation}\label{eq:bondal}
Q = \Bigg(
\begin{tikzcd}
\bullet \arrow[r,bend left,"\alpha_1"] \arrow[r,bend right,swap,"\beta_1"] & \bullet   \arrow[r,bend left,"\alpha_2"] \arrow[r,bend right,swap,"\beta_2"] & \bullet
\end{tikzcd}\Bigg\vert \; \beta_2 \circ \alpha_1 = \alpha_2 \circ \beta_1 = 0 \Bigg) \quad
\end{equation}
Let $kQ$ denote the path algebra on the quiver $Q$ and $I$ the ideal generated by the paths $\langle \beta_2 \alpha_1, \alpha_2\beta_1 \rangle$. Let $kQ/I$ be the corresponding path algebra with relations and $D^b(Q)$ the bounded derived category of right modules over the algebra $kQ/I$. We begin by characterizing and studying the objects which serve as the building blocks of this category.
\begin{lemma}\label{lem:projinj}
The projective and injective right modules over $kQ/I$ are given by the following representations
\begin{gather*}
\begin{tikzcd}
P_1: \mathbb{C} \arrow[r, shift left, "1 \rightarrow (1.0)"] \arrow[r, shift right,"1 \rightarrow (0.1)", labels = below]& \mathbb{C}^{2}\arrow[r, shift left, "(x.y) \rightarrow (x.0)"] \arrow[r, shift right,"(x.y) \rightarrow (0.y)", labels = below]& \mathbb{C}^2
\end{tikzcd}
\quad
\begin{tikzcd}
P_2: 0 \arrow[r, shift left] \arrow[r, shift right]& \mathbb{C}\arrow[r, shift left, "1\rightarrow (1.0)"] \arrow[r, shift right,"1 \rightarrow (0.1)",labels=below]& \mathbb{C}^2
\end{tikzcd}
\quad
\begin{tikzcd}
P_3: 0 \arrow[r, shift left,""] \arrow[r, shift right,"",labels=below]& 0\arrow[r, shift left] \arrow[r, shift right]& \mathbb{C}
\end{tikzcd}
\\
\begin{tikzcd}
I_1 : \mathbb{C} \arrow[r, shift left, ""] \arrow[r, shift right,"", labels = below]& 0\arrow[r, shift left, ""] \arrow[r, shift right,"", labels = below]&0
\end{tikzcd}\quad
\begin{tikzcd}
I_2 : \mathbb{C}^2 \arrow[r, shift left, "(x.y) \rightarrow x"] \arrow[r, shift right, "(x.y) \rightarrow y", labels=below]& \mathbb{C}\arrow[r, shift left, ""] \arrow[r, shift right,"",labels=below]& 0
\end{tikzcd}\quad
\begin{tikzcd}
I_3 : \mathbb{C}^2 \arrow[r, shift left, "(x.y) \rightarrow (x.0)"] \arrow[r, shift right,"(x.y) \rightarrow (0.y)", labels = below]& \mathbb{C}^{2}\arrow[r, shift left, "(x.y) \rightarrow x"] \arrow[r, shift right,"(x.y) \rightarrow y", labels = below]& \mathbb{C}
\end{tikzcd}
\end{gather*}
Moreover, these satisfy the following extension relations:
\begin{enumerate}
\item
$Hom(P_i,P_j) = \mathbb{C}^2$ for $i > j$
\item
$Hom(P_i, P_i) = \mathbb{C}$ for all $i$
\item
$Hom(P_i, P_j) = 0$ for $i < j$
\end{enumerate}
\end{lemma}
In particular, the action of the Serre functor on the projective and injective modules can be described explicitly.
\begin{lemma}\label{lem:serre}
The Serre functor $S$ exists and its action on the projective modules satisfies $S(P_i) \simeq I_i$. Moreover, its action on the morphisms is given by the following:
\[
\begin{tikzcd}[row sep = 0.3cm, column sep = 0.6cm]
0 \arrow[r, shift left] \arrow[r, shift right]\arrow[dd]& 0\arrow[dd] \arrow[r, shift left] \arrow[r, shift right] &\mathbb{C} \arrow[dd, "{\begin{psmallmatrix} 1\\ 0\end{psmallmatrix}}"']& \mathbb{C}^2 \arrow[r, shift left] \arrow[r, shift right]\arrow[dd, "{\begin{psmallmatrix} 1&0\\0&0\end{psmallmatrix}}"]& \mathbb{C}^{2}\arrow[r, shift left] \arrow[r, shift right]\arrow[dd, "{\begin{psmallmatrix} 1& 0\end{psmallmatrix}}"]& \mathbb{C}\arrow[dd]&0 \arrow[r, shift left] \arrow[r, shift right]\arrow[dd]& 0\arrow[dd] \arrow[r, shift left] \arrow[r, shift right] &\mathbb{C} \arrow[dd, "{\begin{psmallmatrix} 0\\ 1\end{psmallmatrix}}"']& \mathbb{C}^2 \arrow[r, shift left] \arrow[r, shift right]\arrow[dd, "{\begin{psmallmatrix} 0&0\\0&1\end{psmallmatrix}}"]& \mathbb{C}^{2}\arrow[r, shift left] \arrow[r, shift right]\arrow[dd, "{\begin{psmallmatrix} 0&1\end{psmallmatrix}}"]& \mathbb{C}\arrow[dd]\\
&&{}\arrow[r,leftrightsquigarrow ]&{}&&&&&{} \arrow[r,leftrightsquigarrow]&{}
 \\
0 \arrow[r, shift left] \arrow[r, shift right]& \mathbb{C}\arrow[r, shift left] \arrow[r, shift right]& \mathbb{C}^2 &  \mathbb{C}^2 \arrow[r, shift left] \arrow[r, shift right]& \mathbb{C}\arrow[r, shift left] \arrow[r, shift right]& 0 &0 \arrow[r, shift left] \arrow[r, shift right]& \mathbb{C}\arrow[r, shift left] \arrow[r, shift right]& \mathbb{C}^2 &  \mathbb{C}^2 \arrow[r, shift left] \arrow[r, shift right]& \mathbb{C}\arrow[r, shift left] \arrow[r, shift right]& 0 
\\
\\
0 \arrow[r, shift left] \arrow[r, shift right]\arrow[dd]& \mathbb{C}\arrow[dd,"{\begin{psmallmatrix} 1\\ 0\end{psmallmatrix}}"'] \arrow[r, shift left] \arrow[r, shift right] &\mathbb{C}^2 \arrow[dd, "{\begin{psmallmatrix} 1& 0\\0&0\end{psmallmatrix}}"']& \mathbb{C}^2 \arrow[r, shift left] \arrow[r, shift right]\arrow[dd, "{\begin{psmallmatrix} 1&0\end{psmallmatrix}}"]& \mathbb{C}\arrow[r, shift left] \arrow[r, shift right]\arrow[dd]& 0\arrow[dd]&0 \arrow[r, shift left] \arrow[r, shift right]\arrow[dd]& \mathbb{C}\arrow[dd,"{\begin{psmallmatrix} 0\\ 1\end{psmallmatrix}}"'] \arrow[r, shift left] \arrow[r, shift right] &\mathbb{C}^2 \arrow[dd, "{\begin{psmallmatrix} 0&0\\0&1\end{psmallmatrix}}"']& \mathbb{C}^2 \arrow[r, shift left] \arrow[r, shift right]\arrow[dd, "{\begin{psmallmatrix} 0&1\end{psmallmatrix}}"]&\mathbb{C}\arrow[r, shift left] \arrow[r, shift right]\arrow[dd]& 0\arrow[dd]\\
&&{}\arrow[r,leftrightsquigarrow ]&{}&&&&&{} \arrow[r,leftrightsquigarrow]&{}
 \\
\mathbb{C} \arrow[r, shift left] \arrow[r, shift right]& \mathbb{C}^2\arrow[r, shift left] \arrow[r, shift right]& \mathbb{C}^2 & \mathbb{C} \arrow[r, shift left] \arrow[r, shift right]& 0\arrow[r, shift left] \arrow[r, shift right]& 0 &\mathbb{C} \arrow[r, shift left] \arrow[r, shift right]& \mathbb{C}^2\arrow[r, shift left] \arrow[r, shift right]& \mathbb{C}^2 &  \mathbb{C} \arrow[r, shift left] \arrow[r, shift right]& 0\arrow[r, shift left] \arrow[r, shift right]& 0 
\end{tikzcd}
\]
where in the first row, we have presented a basis of morphisms of $Hom(P_3,P_2)$ and its Serre dual, and in the second row, we have presented a basis of morphisms of $Hom(P_2,P_1)$ and its Serre dual.
\end{lemma}
\begin{proof}
This is standard, see for example \cite[Proposition 2.29]{schiffler2014quiver} for the explicit action on the morphisms between projective objects. 
\end{proof}
In addition to the projective and injective objects, there is a canonical quiver representation, $P$, which is an exceptional object in $D^b(Q)$. Moreover, it is a $(4,2)$-fractional Calabi-Yau object under the action of the Serre functor, i.e. $S^2P = P[4]$, and is responsible for the various interesting phenomena associated to this category.

\begin{lemma}\label{lem:p}
The following quiver representations
\[
\begin{tikzcd}
P: \mathbb{C} \arrow[r, shift left, "1"] \arrow[r, shift right,"0", labels = below]& \mathbb{C}\arrow[r, shift left, "1"] \arrow[r, shift right,"0", labels = below]& \mathbb{C}
\end{tikzcd}
\quad
\begin{tikzcd}
\tilde{P}: \mathbb{C} \arrow[r, shift left, "0"] \arrow[r, shift right,"1", labels = below]& \mathbb{C}\arrow[r, shift left, "0"] \arrow[r, shift right,"1", labels = below]& \mathbb{C}
\end{tikzcd}
\]
are exceptional objects in $D^b(Q)$. In particular, there is an isomorphism $S(P) \simeq \tilde{P}[2]$.
\end{lemma}
\begin{proof}
Follows from a direct computation using the results of Lemma~\ref{lem:projinj} and \ref{lem:serre}.
\end{proof}
As a consequence of this object, $D^b(Q)$ does not satisfy the Jordan-H\"{o}lder property~\cite{kuznetsov2013simple}.
\begin{corollary}\label{cor:nonequiv}
The category $D^b(Q)$ admits two semi-orthogonal decompositions given by the following.
\[
D^b(Q) = \langle P_3, P_2, P_1 \rangle = \langle P^\perp, P \rangle
\]
where $P^\perp = \{T \in D^b(Q)\, \vert\, Hom(T, P[i]) = 0\text{ for all }i \in \mathbb{Z} \}$. 
\end{corollary}
As a consequence of the fractional Calabi-Yau property of the object $P \in D^b(Q)$, this induces a numerically trivial $3$-spherical object $E$ in the subcategory $P^\perp$. In the subsequent sections of this paper, we will see that the object $E$ is responsible for all the interesting structural phenomena of the category $P^\perp$.
\begin{lemma}\label{lem:E}
Let $E$ be the object generated from the extension
\begin{equation}\label{eq:E}
\begin{tikzcd}
S(P) [-1] \arrow[r]& E \arrow[r]& P
\end{tikzcd}
\end{equation}
Then $E$ is a $3$-spherical object in the category $P^\perp$. Moreover, the class $[E] \in K_0(Q)$ is trivial in the Grothendieck group.
\end{lemma}
\begin{proof}
We first prove that $E$ is numerically trivial. We first observe that $[P] = [\tilde{P}]$ from their projective resolutions. We have the following
\[
[E] = [S(P)[-1]] + [P] = [\tilde{P}[1]] + [P] = - [P] + [P] = 0
\]
where the second equality follows from Lemma~\ref{lem:p}. Thus, the claim follows.

We now prove the inclusion $E \in P^\perp$. Taking the long exact sequence associated to the functor $Hom(P, \bullet)$ on sequence~\eqref{eq:E} and applying Lemma~\ref{lem:p}, we obtain the exact sequence
\[
0 \rightarrow Hom(P,E) \rightarrow Hom(P,P) = \mathbb{C} \rightarrow Hom(P, \tilde{P}[2]) = \mathbb{C} \rightarrow Hom(P,E[1]) \rightarrow 0
\]
where all other vector spaces are zero. But the morphism $Hom(P,P) \rightarrow Hom(P,\tilde{P}[2])$ must be nonzero as it is defined by a composition of the identity morphism with the nontrivial morphism $P \rightarrow \tilde{P}[2]$. Thus, it must be an isomorphism and we have that $Hom^\bullet(P,E) = 0$.

We now prove that $E$ is a $3$-spherical object in $P^\perp$. We first prove that $S_{P^\perp}(E) \simeq E[3]$. By \cite[Lemma 2.7]{Kuznetsov_2019}, we have a natural isomorphism $S_{P^\perp}^{-1} \simeq \mathbb{L}_P \circ S_Q^{-1}$. Applying the composition $\mathbb{L}_P \circ S_Q^{-1}$ to sequence~\eqref{eq:E}, we obtain the triangle
\[
\begin{tikzcd}
0 \arrow[r] & S_{P^\perp}^{-1} E \arrow[r, "\sim"] & \mathbb{L}_P \tilde{P}[-2]
\end{tikzcd}
\]
On the other hand, applying the functor $\mathbb{L}_P$ on sequence~\eqref{eq:E}, we obtain the triangle
\[
\begin{tikzcd}
\mathbb{L}_P \tilde{P}[1] \arrow[r, "\sim"] & E \arrow[r]& 0
\end{tikzcd}
\]
where we have used that $\mathbb{L}_PE \simeq E$ as $E \in P^\perp$. Combining the above isomorphisms, we have that $S_{P^\perp}^{-1}E \simeq E[-3]$, and the claim follows.

Finally, we prove that the extension algebra satisfies $Ext^\bullet(E,E) = (1,0,0,1)$. By \cite[Proposition 2.4]{oka2006}, there exists a spectral sequence
\[
E_2^{p,q} \coloneqq \bigoplus\limits_i Hom^p(\mathcal{H}_Q^i(E), \mathcal{H}_Q^{i+q}(E)) \implies Hom^{p+q}(E,E)
\]
where $\mathcal{H}_Q$ denotes taking cohomology with respect to the canonical heart of a bounded t-structure $\mathrm{mod}\text{-}kQ/I \subset D^b(Q)$. By sequence~\eqref{eq:E}, we have that $\mathcal{H}^{-1}(E) \simeq \tilde{P}$ and $\mathcal{H}^0(E) \simeq P$, and the claim follows from a straightforward computation.
\end{proof}

In the rest of this section, we will study the category $P^\perp$ in detail. Specifically, we study the projection of the generating objects of $D^b(Q)$ onto the subcategory $P^\perp$ and their homological properties. Two canonical objects in this category are given by the following.
\begin{lemma}\label{lem:cd}
The following quiver representations
\[
\begin{tikzcd}
D\colon0 \arrow[r, shift left] \arrow[r, shift right]& \mathbb{C}\arrow[r, shift left, "1"] \arrow[r, shift right,"0", labels = below]& \mathbb{C}
\end{tikzcd}
\quad
\begin{tikzcd}
\tilde{C} \colon \mathbb{C} \arrow[r, shift left, "0"] \arrow[r, shift right,"1", labels = below]& \mathbb{C} \arrow[r,shift left] \arrow[r,shift right]& 0
\end{tikzcd}
\]
are contained in the category $P^\perp$.
\end{lemma}
\begin{proof}
We prove the claim for the object $D$. By Lemma~\ref{lem:projinj} and using the projective resolutions for $P$ and $D$, the Euler form 
\[\chi \colon K_0(Q) \times K_0(Q) \rightarrow \mathbb{Z}\] satisfies $\chi(P,D) = 0$. As $P$ is of projective dimension $2$, we have $Hom^i(P,D) = 0$ for all $i > 2$ and $i < 0$. But clearly, we have that $Hom(P,D) = 0$ and $Hom(P,D[2]) = Hom(D[2], \tilde{P}[2]) = Hom(D, \tilde{P}) = 0$ where we applied Lemma~\ref{lem:p} in the second chain of equalities along with Serre duality. Together with the vanishing of the Euler form, it follows that $Hom(P,D[1]) = 0$ and we conclude.

The argument for the object $\tilde{C}$ is identical.
\end{proof}
As a consequence of Lemma~\ref{lem:cd}, we note that the object $A \coloneqq Cone(D \rightarrow \tilde{C})$ is also contained in the category $P^\perp$. For a systematic study of the category $P^\perp$, it is useful to understand the semi-orthogonal decompositions of the projective and injective objects. We record these in the following two Lemmas.
\begin{lemma}\label{lem:injective}
The injective objects $I_i$ admit the following decompositions with respect to the semiorthogonal decomposition $D^b(Q) = \langle P^\perp, P\rangle$.
\[
\begin{tikzcd}[row sep=tiny]
P \arrow[r] & I_3 \arrow[r] & \tilde{C}\\[-5pt]
P \arrow[r] & I_2 \arrow[r] & A \\[-5pt]
P \arrow[r] & I_1 \arrow[r] & D[1]
\end{tikzcd}
\]
\end{lemma}
\begin{proof}
The first sequence is obviously exact in the abelian category $\mathrm{mod}\text{-}kQ/I$.
The second comes from the quasi-isomorphism of complexes
\[
\begin{tikzcd}
0 \arrow[r, shift left] \arrow[r, shift right]& \mathbb{C}\arrow[d,"1"]\arrow[r, shift left, "1"] \arrow[r, shift right,"0", labels = below]& \mathbb{C} \arrow[d,"1"]\arrow[r] & \mathbb{C}\arrow[d,"{\begin{psmallmatrix}0\\1 \end{psmallmatrix}}"'] \arrow[r, shift left, "0"] \arrow[r, shift right,"1", labels = below]& \mathbb{C}\arrow[d,"1"]\arrow[r,shift left] \arrow[r,shift right]& 0\\
\mathbb{C} \arrow[r, shift left, "1"] \arrow[r, shift right,"0", labels = below]& \mathbb{C}\arrow[r, shift left, "1"] \arrow[r, shift right,"0", labels = below]& \mathbb{C} \arrow[r] & \mathbb{C}^2 \arrow[r, shift left, "(x.y) \rightarrow x"] \arrow[r, shift right, "(x.y) \rightarrow y", labels=below]& \mathbb{C}\arrow[r, shift left, ""] \arrow[r, shift right,"",labels=below]& 0
\end{tikzcd}
\]
identifying $Cone(P \rightarrow I_2) \xrightarrow{\sim} Cone(D \rightarrow \tilde{C}) = A$.
The third sequence comes from the obvious exact sequence
\[
\begin{tikzcd}
D \arrow[r] & P \arrow[r] & I_1
\end{tikzcd}
\]
in the abelian category $\mathrm{mod}\text{-}kQ/I$. 
\end{proof}
\begin{lemma}\label{lem:projective}
Similarly, the projective objects $P_i$ admit the following decompositions with respect to the semiorthogonal decomposition $D^b(Q) = \langle S(P), P^\perp\rangle$.
\[
\begin{tikzcd}[row sep = tiny]
\tilde{C} \arrow[r] & P_3[1] \arrow[r] & \tilde{P}[1]\\[-5pt]
A \arrow[r] & P_2[1] \arrow[r] & \tilde{P}[1]\\[-5pt]
D[1] \arrow[r] & P_1[1] \arrow[r] & \tilde{P}[1] \\
\end{tikzcd}
\]
\end{lemma}
\begin{proof}
The first sequence follows from observing the exact sequence
\[
\begin{tikzcd}
P_3 \arrow[r] & \tilde{P} \arrow[r] & \tilde{C}
\end{tikzcd}
\]
in $\mathrm{mod}\text{-}kQ/I$. The second comes from the quasi-isomorphism of complexes
\[
\begin{tikzcd}
0 \arrow[r, shift left] \arrow[r, shift right] & \mathbb{C}\arrow[r, shift left, "1\rightarrow (1.0)"]\arrow[d,"1"] \arrow[r, shift right,"1 \rightarrow (0.1)",labels=below]&\mathbb{C}^2 \arrow[d,"{\begin{psmallmatrix}1 & 0 \end{psmallmatrix}}"]\arrow[r] & \mathbb{C}  \arrow[d,"1"] \arrow[r, shift left,"0"] \arrow[r, shift right,"1"'] & \mathbb{C} \arrow[d,"1"]  \arrow[r, shift left,"0"] \arrow[r, shift right,"1"'] & \mathbb{C} 
\\
0 \arrow[r, shift left] \arrow[r, shift right]& \mathbb{C}\arrow[r, shift left, "1"] \arrow[r, shift right,"0", labels = below]& \mathbb{C}\arrow[r] & \mathbb{C} \arrow[r, shift left, "0"] \arrow[r, shift right,"1", labels = below]& \mathbb{C}\arrow[r,shift left] \arrow[r,shift right]& 0
\end{tikzcd}
\]
identifying $A = Cone(D \rightarrow \tilde{C}) \xrightarrow{\sim} Cone(P_2 \rightarrow \tilde{P})$.
The third comes from the obvious exact sequence
\[
\begin{tikzcd}
D \arrow[r] & P_1 \arrow[r] & \tilde{P}
\end{tikzcd}
\]
in $\mathrm{mod}\text{-}kQ/I$.
\end{proof}
We conclude this section by recording some homological properties of the projections $\mathbb{L}_PI_i \in P^\perp$ obtained in Lemma~\ref{lem:injective}.
\begin{lemma}\label{lem:extensions}
The extension algebras of the objects $\tilde{C}, A, D \in P^\perp$ with $P$ are given by the following.
\begin{enumerate}
\item
$Hom^\bullet(\tilde{C},P) = (0,1,1)$
\item
$Hom^\bullet(A,P) = (0,1,1)$
\item
$Hom^\bullet(D,P) = (1,1,0)$
\end{enumerate}
\end{lemma}
\begin{proof}
This follows from the long exact sequence associated to the functor $Hom(\bullet,P)$ applied to the following sequences for each of the above cases.
\begin{enumerate}
\item
$P_3 \rightarrow \tilde{P} \rightarrow \tilde{C}$.
\item
$P \rightarrow I_2 \rightarrow A$.
\item
$D \rightarrow P \rightarrow I_1$.
\end{enumerate}
\end{proof}

For later use, we also record the following resolutions of the objects $A, \tilde{C}, D \in P^\perp$. 
\begin{lemma}\label{lem:aresolution}
$A$ admits the following projective and injective resolutions $P_A^\bullet,I_A^\bullet$
\[
P^\bullet \simeq \begin{tikzcd}
0 \arrow[r, shift left, ""] \arrow[r, shift right, "", labels = below] \arrow[d, "", labels = left] & 0\arrow[r, shift left] \arrow[r, shift right] \arrow[d]& \mathbb{C} \arrow[d,"{\begin{psmallmatrix} 0 & 1 & 0 & 0\end{psmallmatrix}}",labels=right] \\
0\arrow[r, shift left, ""] \arrow[r, shift right,"", labels = below]\arrow[d]& \mathbb{C}^2 \arrow[d,"id",labels=right]\arrow[r, shift left,""] \arrow[r, shift right, labels = below]& \mathbb{C}^4 \arrow[d, "{\begin{psmallmatrix} 1 & 0 & 0 & 0\\ 0 & 0 & 0 & 1\end{psmallmatrix}}"] \\
\mathbb{C} \arrow[r, shift left] \arrow[r, shift right]& \mathbb{C}^2 \arrow[r, shift left] \arrow[r, shift right]& \mathbb{C}^2\\
\end{tikzcd}
\quad
I^\bullet \simeq
\begin{tikzcd}
\mathbb{C}^2 \arrow[r, shift left, ""] \arrow[r, shift right, "", labels = below] \arrow[d, "{\begin{psmallmatrix}1& 0\\ 0 & 0\\0 & 0\\0 &1\end{psmallmatrix}}", labels = left] & \mathbb{C}^2 \arrow[r, shift left] \arrow[r, shift right] \arrow[d, "id"]& \mathbb{C} \arrow[d] \\
\mathbb{C}^4 \arrow[r, shift left, ""] \arrow[r, shift right,"", labels = below]\arrow[d, shift left,"{\begin{psmallmatrix}0 \\ 0 \\ 1 \\0\end{psmallmatrix}}", labels=left]& \mathbb{C}^2\arrow[d]\arrow[r, shift left] \arrow[r, shift right, labels = below]& 0\arrow[d] \\
\mathbb{C} \arrow[r, shift left] \arrow[r, shift right]& 0 \arrow[r, shift left] \arrow[r, shift right]& 0\\
\end{tikzcd}
\]
where the morphisms of complexes are given by the following
\[
\begin{tikzcd}
0 \arrow[r, shift left] \arrow[r, shift right] \arrow[d] & \mathbb{C}\arrow[d, "{\begin{psmallmatrix}1 \\ 0\end{psmallmatrix}}", labels = left]\arrow[r, shift left, "1", labels=above] \arrow[r, shift right, "0", labels=below] & \mathbb{C} \arrow[d] \arrow[r] & \mathbb{C} \arrow[r, shift left, "", labels=above] \arrow[r, shift right, "", labels=below] \arrow[d, "{\begin{psmallmatrix}0 \\ 1 \\ 0 \\ 0\end{psmallmatrix}}", labels = right]& \mathbb{C} \arrow[d, "{\begin{psmallmatrix}1\\ 0\end{psmallmatrix}}"] \arrow[r, shift left] \arrow[r, shift right] & 0\arrow[d]\\
\mathbb{C}^2 \arrow[r, shift left] \arrow[r, shift right] & \mathbb{C}^2 \arrow[r, shift left] \arrow[r, shift right] & \mathbb{C} \arrow[r] & \mathbb{C}^4 \arrow[r, shift left] \arrow[r, shift right] & \mathbb{C}^2 \arrow[r, shift left] \arrow[r, shift right] & 0 \arrow[r] & \mathbb{C} \arrow[r, shift left] \arrow[r, shift right] & 0 \arrow[r, shift left] \arrow[r, shift right] & 0\\[-10pt]
0 \arrow[r, shift left] \arrow[r, shift right] & 0\arrow[r, shift left] \arrow[r, shift right] & \mathbb{C} \arrow[r] &0\arrow[d] \arrow[r, shift left] \arrow[r, shift right] & \mathbb{C}^2 \arrow[d, "{\begin{psmallmatrix}0 \\1 \end{psmallmatrix}}", labels = left]\arrow[r, shift left] \arrow[r, shift right] & \mathbb{C}^4 \arrow[d, "{\begin{psmallmatrix}0 \\ 0 \\ 1 \\ 0\end{psmallmatrix}}", labels=left]\arrow[r] & \mathbb{C} \arrow[d, "id"] \arrow[r, shift left] \arrow[r, shift right] & \mathbb{C}^2 \arrow[d, "{\begin{psmallmatrix}0 \\ 1\end{psmallmatrix}}"]\arrow[r, shift left] \arrow[r, shift right] & \mathbb{C}^2\arrow[d]\\
&  &  & 0 \arrow[r, shift left] \arrow[r, shift right] & \mathbb{C}\arrow[r, shift left, "1", labels=above] \arrow[r, shift right, "0", labels=below] & \mathbb{C} \arrow[r] & \mathbb{C} \arrow[r, shift left, "", labels=above] \arrow[r, shift right, "", labels=below]& \mathbb{C} \arrow[r, shift left] \arrow[r, shift right] & 0
\end{tikzcd}
\]
\end{lemma}
\begin{proof}
Follows from a direct computation.
\end{proof}
\begin{lemma}\label{lem:cdresolution}
$\tilde{C}$ and $D$ admit the following projective and injective resolutions respectively.
\[
P_{\tilde{C}}^\bullet \simeq \begin{tikzcd}
0 \arrow[r, shift left, ""] \arrow[r, shift right, "", labels = below] \arrow[d, "", labels = left] & 0\arrow[r, shift left] \arrow[r, shift right] \arrow[d]& \mathbb{C} \arrow[d,"{\begin{psmallmatrix} 0 \\ 1 \\ 0\end{psmallmatrix}}",labels=right] \\
0\arrow[r, shift left, ""] \arrow[r, shift right,"", labels = below]\arrow[d]& \mathbb{C} \arrow[d,"{\begin{psmallmatrix}1\\0 \end{psmallmatrix}}"']\arrow[r, shift left,"{\begin{psmallmatrix}1& 0 & 0 \end{psmallmatrix}}"] \arrow[r, shift right,"{\begin{psmallmatrix}0& 1 & 0 \end{psmallmatrix}}", labels = below]& \mathbb{C}^3 \arrow[d, "{\begin{psmallmatrix} 1 & 0 & 0 \\ 0 & 0 &  1\end{psmallmatrix}}"] \\
\mathbb{C} \arrow[r, shift left] \arrow[r, shift right]& \mathbb{C}^2 \arrow[r, shift left] \arrow[r, shift right]& \mathbb{C}^2\\
\end{tikzcd}
\quad
I_{D}^\bullet \simeq
\begin{tikzcd}
\mathbb{C}^2 \arrow[r, shift left, ""] \arrow[r, shift right, "", labels = below] \arrow[d, "{\begin{psmallmatrix}0& 0\\ 1 & 0\\0 & 1\end{psmallmatrix}}", labels = left] & \mathbb{C}^2 \arrow[r, shift left] \arrow[r, shift right] \arrow[d, "{\begin{psmallmatrix}0& 1\end{psmallmatrix}}"]& \mathbb{C} \arrow[d] \\
\mathbb{C}^3 \arrow[r, shift left, "{\begin{psmallmatrix}1& 0 & 0 \end{psmallmatrix}}"] \arrow[r, shift right,"{\begin{psmallmatrix}0&0&1 \end{psmallmatrix}}", labels = below]\arrow[d, shift left,"{\begin{psmallmatrix}1\\0\\0\end{psmallmatrix}}", labels=left]& \mathbb{C}\arrow[d]\arrow[r, shift left] \arrow[r, shift right, labels = below]& 0\arrow[d] \\
\mathbb{C} \arrow[r, shift left] \arrow[r, shift right]& 0 \arrow[r, shift left] \arrow[r, shift right]& 0\\
\end{tikzcd}
\]
\end{lemma}
\begin{proof}
Follows from a direct computation.
\end{proof}
\section{The Serre functor}\label{sec:serre}
In the following $Q$ will denote the Bondal quiver~\eqref{eq:bondal}, and $D^b(Q)$ the bounded derived category of right modules over the path algebra with relations $kQ/I$ associated to the quiver $Q$, with Serre functor $S$. We study the admissible subcategory $P^\perp$ obtained from the semi-orthogonal decomposition $D^b(Q) = \langle P^\perp, P \rangle$. In this section, we prove the following.
\begin{theorem}\label{thm:serremutation}
There exists a natural isomorphism of functors 
\[S{-1}[1]|_{P^\perp} \simeq \mathbb{L}_{S(P)} \colon P^\perp \rightarrow \prescript{\perp}{}{P}
\]
\end{theorem}
As a corollary, we obtain the following characterization of the Serre functor on the admissible subcategory $P^\perp$.
\begin{corollary}\label{cor:serrepperp}
The Serre functor on the category $P^\perp$ satisfies the following: $S^{-1}_{P^\perp} \simeq \mathbb{T}_E[-1]$.
\end{corollary}
\begin{proof}
We simply apply \cite[Theorem 3.11]{halpernleistner2016autoequivalences} to our case. By \cite[Lemma 2.7]{Kuznetsov_2019} and theorem~\ref{thm:serremutation}, we have
\[
S_{P^\perp}^{-1} \simeq \mathbb{L}_P \circ S^{-1} \simeq \mathbb{L}_P\circ \mathbb{L}_{S(P)}[-1]
\]
We have the following semi-orthogonal decompositions
\[
D^b(Q) = \langle S(P), P^\perp\rangle = \langle P^\perp, P \rangle = \langle P, \prescript{\perp}{}{P} \rangle = \langle \prescript{\perp}{}{P} , S(P) \rangle
\]
By sequence~\eqref{eq:E}, we have that $\mathbb{L}_P(S(P)[-1]) \simeq E$. Applying \cite[Theorem 3.11]{halpernleistner2016autoequivalences} with $\mathcal{A} = \langle S(P) \rangle, \mathcal{B} = P^\perp, \mathcal{A}' = \langle P \rangle, \mathcal{B}' = \prescript{\perp}{}{P}$, we obtain $\mathbb{L}_P\circ \mathbb{L}_{S(P)}[-1] \simeq \mathbb{T}_E[-1]$ and we conclude.
\end{proof}
As a consequence of the characterization of the Serre functor, we obtain a classification of spherical objects in $P^\perp$.
\begin{corollary}\label{cor:unique}
Assume that $F \in P^\perp$ is an $n$-spherical object. Then either there exists an isomorphism $F \simeq E[k]$ for some integer $k$ or it must be the case that $n = 1$.
\end{corollary}
\begin{proof}
By definition, we have the sequence
\[
\begin{tikzcd}
RHom(E,F) \otimes E \arrow[r] & F \arrow[r] & \mathbb{T}_E(F)
\end{tikzcd}
\]
By corollary~\ref{cor:serrepperp}, we have an isomorphism $\mathbb{T}_E(F) \simeq S_{P^\perp}^{-1}(F)[1] \simeq F[-n+1]$. If $n\neq 1$, then the second morphism in the above sequence must be zero by definition of a spherical object. Thus, there exists an isomorphism $RHom(E,F) \otimes E \simeq F \oplus \mathbb{T}_E(F)[-1]$. But $D^b(Q)$ is Krull-Schmidt, indeed for any finite dimensional $k$-algebra $A$, the derived category $D^b(\mathrm{mod}\text{-}A)$ is Krull-Schmidt. In particular, this must be true for any full subcategory and so $F$ must be isomorphic to $E$ up to shifts.
\end{proof}
\begin{corollary}\label{cor:serreinv}
There exists no Serre-invariant stability conditions on $P^\perp$. 
\end{corollary}
\begin{proof}
This follows directly from Corollary 6.9 and Proposition 6.17 in~\cite{kuznetsov2021serre}.
\end{proof}
We note that Corollary~\ref{cor:serreinv} also follows directly from the observation that there exists spherical objects of different dimensions in $P^\perp$, without the computation of the Serre functor. The result then follows from \cite[Lemma 6.3]{2019arXiv190109461E} applied to these spherical objects.

The strategy of the proof of theorem~\ref{thm:serremutation} is summarized by the following sequence of reductions.
\begin{enumerate}
\item
In Section~\ref{sec:extension}, we extend the left mutation functor $\mathbb{L}_{S(P)} \colon P^\perp \rightarrow \prescript{\perp}{}{P}$ to an equivalence $\Phi \colon D^b(Q) = \langle P^\perp, P \rangle \simeq D^b(Q) = \langle \prescript{\perp}{}{P}, S^{-1}(P)[1] \rangle$. More precisely, we establish the following:
\begin{replemma}{lem:extension}
There exists an exact functor $\Phi \colon D^b(Q) \rightarrow D^b(Q)$ satisfying $\restr{\Phi}{P^\perp} \simeq \restr{\mathbb{L}_{S(P)}}{P^\perp}$ and $\Phi(P) \simeq S^{-1}(P)[1]$. Moreover, $\Phi$ is an equivalence of categories.
\end{replemma}
\item
In Section~\ref{sec:extensioniso}, we reduce the assertion of the existence of a natural isomorphism $\Phi \simeq S^{-1}[1]$ to the full subcategory $\mathcal{I}$ on a representative set of injective objects. More precisely, we prove the following:
\begin{replemma}{lem:extensioniso}
Assume $F_1, F_2 \colon D^b(Q) \rightarrow D^b(Q)$ equivalences of categories such that there exists a natural isomorphism $F_1|_{\mathcal{I}} \simeq F_2|_{\mathcal{I}}$ on the restrictions. Then there exists a natural isomorphism $F_1 \simeq F_2$.
\end{replemma}
\item
In Section~\ref{sec:injectives}, we reduce the existence of the natural isomorphism on the injective objects to establishing the existence on their restrictions to $P^\perp$. More precisely,
\begin{replemma}{lem:reduction}
Assume that there exists isomorphisms $\eta_{P^\perp} \colon \Phi\mathbb{L}_P I_i \simeq S^{-1}\mathbb{L}_P I_i[1]$ for all $i,j$ such that the following diagram commutes for any morphism $f \colon I_i \rightarrow I_j$.
\[
\begin{tikzcd}
\Phi(\mathbb{L}_PI_i) \arrow[r,"\Phi \mathbb{L}_P f"]\isoarrow{d}& \Phi(\mathbb{L}_PI_j) \isoarrow{d}\\
S^{-1}(\mathbb{L}_PI_i)[1] \arrow[r,"S^{-1}{[1]}(\mathbb{L}_Pf)"] &S^{-1}(\mathbb{L}_PI_j)[1]
\end{tikzcd}
\]
Then there exists isomorphisms $\eta_i \colon \Phi I_i \simeq S^{-1}I_i[1]$ for all $i,j$ such that the following diagram commutes.
\[
\begin{tikzcd}
\Phi(I_i) \arrow[r,"\Phi(f)"]\isoarrow{d}& \Phi(I_j) \isoarrow{d}\\
S^{-1}(I_i)[1] \arrow[r,"S^{-1}f{[1]}"] &S^{-1}(I_j)[1]
\end{tikzcd}
\]
\end{replemma}
The utility of this Lemma is to allow one to remain agnostic to the details of the gluing functor obtained in Lemma~\ref{lem:extension}.
\item
Finally in Section~\ref{sec:compatibility}, we prove the existence of the natural isomorphism on the restrictions of the injective objects to the category $P^\perp$. More precisely,
\begin{replemma}{lem:computation}
There exists isomorphisms $\eta_{P_i} \colon \Phi\mathbb{L}_P I_i \simeq S^{-1}\mathbb{L}_P I_i[1]$ for all $i,j$ such that the following diagram commutes for any morphism $f \colon I_i \rightarrow I_j$.
\[
\begin{tikzcd}
\Phi(\mathbb{L}_PI_i) \arrow[r,"\Phi \mathbb{L}_P f"]\isoarrow{d}& \Phi(\mathbb{L}_PI_j) \isoarrow{d}\\
S^{-1}(\mathbb{L}_PI_i)[1] \arrow[r,"S^{-1}{[1]}(\mathbb{L}_Pf)"] &S^{-1}(\mathbb{L}_PI_j)[1]
\end{tikzcd}
\]
\end{replemma}
\end{enumerate}
\subsection{Extension of the mutation functor}\label{sec:extension}
In this section, we prove the following:
\begin{lemma}\label{lem:extension}
There exists an exact functor $\Phi \colon D^b(Q) \rightarrow D^b(Q)$ satisfying $\restr{\Phi}{P^\perp} \simeq \restr{\mathbb{L}_{S(P)}}{P^\perp}$ and $\Phi(P) \simeq S^{-1}(P)[1]$. Moreover, $\Phi$ is an equivalence of categories.
\end{lemma}
To prove the Lemma, we will use the extension results established in \cite[Proposition 2.5]{li2020refined}. In order to identify with their setting, we will need to use the following geometric properties of $D^b(Q)$. 
\begin{prop}[\cite{kuznetsov2013simple}]\label{prop:kuznetsovembed}
There exists an admissible embedding $D^b(Q) \xhookrightarrow{} D^b(X)$ where $X$ is the sequential blowup of $\mathbb{P}^3$ along two rational curves.
\end{prop}
\begin{corollary}\label{cor:leftmutfm}
The mutation functors $\mathbb{L}_{S(P)},\mathbb{L}_{P} \colon D^b(Q) \rightarrow D^b(Q)$ are of Fourier-Mukai type.
\end{corollary}
\begin{proof}
Indeed, the functor $\mathbb{L}_{S(P)}$ is nothing but the projection functor of $D^b(Q) = \langle \prescript{\perp}{}{P}, S(P) \rangle$ onto the admissible subcategory $\prescript{\perp}{}{P}$. In particular, the composition of projection functors $D^b(X) \rightarrow D^b(Q) \rightarrow \prescript{\perp}{}{P}$ is a projection functor, where the first exists by Proposition~\ref{prop:kuznetsovembed}. By \cite[Theorem 7.1]{Kuznetsov_2011}, the composition is of Fourier-Mukai type. The claim for $\mathbb{L}_P$ follows from an identical argument.
\end{proof}
To use the gluing result, we will also need the following in order to verify assumption~(2) of  Proposition~\ref{prop:fmequiv}.
\begin{lemma}\label{lem:sinverse}
There is an isomorphism $\mathbb{L}_{S(P)}(E) \simeq S^{-1}(E)[1]$.
\end{lemma}
\begin{proof}
Clearly, $\mathbb{L}_{S(P)}(E) \simeq \mathbb{L}_{S(P)}(P)$ by sequence~\eqref{eq:E}. Hence, this fits into an exact triangle
\[
\begin{tikzcd}
\text{RHom}(S(P),P)\otimes S(P) \arrow[r]& P \arrow[r]& \mathbb{L}_{S(P)}(E)
\end{tikzcd}
\]
By Lemma~\ref{lem:p}, we see that $RHom(S(P),P) = RHom(P,P[4]) = \mathbb{C}[-4]$ which yields the exact triangle
\[
\begin{tikzcd}
S(P)[-4] = \tilde{P}[-2] \arrow[r]& P \arrow[r]& \mathbb{L}_{S(P)}(E)
\end{tikzcd}
\]
On the other hand, applying the inverse Serre functor on sequence~\eqref{eq:E} and rotating, we obtain the following
\[
\begin{tikzcd}
\tilde{P}[-2] \arrow[r] & P \arrow[r] & S^{-1}(E)[1]
\end{tikzcd}
\]
As $Hom(\tilde{P}[-2],P) = \mathbb{C}$, there exists an isomorphism between the above two sequences and we conclude $\mathbb{L}_{S(P)}(E) \simeq S^{-1}(E)[1]$.
\end{proof}
In the setting of the following Proposition, $X$ and $Y$ will denote smooth projective varieties over a field $k$ together with admissible embeddings $\alpha \colon \mathcal{A} \xhookrightarrow{} D^b(X)$, $\beta \colon \mathcal{B} \xhookrightarrow{} D^b(Y)$. We denote the left and right adjoints with $\alpha^*, \beta^*$ and $\alpha^!, \beta^!$, respectively. 
\begin{prop}[\cite{li2020refined}, Proposition 2.5]\label{prop:fmequiv}
Let $E \in \prescript{\perp}{}{\mathcal{A}}$ be an exceptional object with counit of adjunction $\eta \colon \alpha \alpha^! (E) \rightarrow E$. Let $\Phi_{\mathcal{E}} \colon D^b(X) \rightarrow D^b(Y)$ be a Fourier-Mukai functor such that $\Phi_\mathcal{E}(\prescript{\perp}{}{\mathcal{A}}) \simeq 0$. Assume in addition the following:
\begin{enumerate}
\item
$\Phi_{\mathcal{E}}|_\mathcal{A}$ is an equivalence onto an admissible subcategory $\mathcal{B}$ with embedding $\beta \colon \mathcal{B} \xhookrightarrow{} D^b(Y)$.
\item
There exists an exceptional object $F$ and an isomorphism $\rho \colon \Phi(\alpha \alpha^!(E)) \simeq \beta \beta^!(F)$.
\end{enumerate}
Then there exists a Fourier-Mukai functor $\Phi' \colon D^b(X) \rightarrow D^b(Y)$ satisfying the following:
\begin{enumerate}
\item
There exists a natural isomorphism $\Phi'|_{\mathcal{A}} \simeq \Phi|_{\mathcal{A}}$.
\item
There exists an isomorphism $\Phi'(E) \simeq F$.
\end{enumerate}
which restricts to an equivalence of categories $\Phi' \colon \langle \mathcal{A}, E \rangle \isomto \langle \mathcal{B} ,F \rangle$.
\end{prop}
Applying this Proposition to our setting of interest, we now prove Lemma~\ref{lem:extension}.
\begin{proof}[Proof of Lemma~\ref{lem:extension}]
We will extend the left mutation functor $\mathbb{L}_{S(P)} \colon P^\perp \isomto \prescript{\perp}{}P$.
We have the admissible embeddings
\begin{align*}
\alpha &\colon P^\perp \xhookrightarrow{\alpha_1} D^b(Q) = \langle P^\perp, P \rangle \xhookrightarrow{\alpha_2} D^b(X) \\
\beta &\colon \prescript{\perp}{}{P} \xhookrightarrow{\beta_1} D^b(Q) = \langle \prescript{\perp}{}{P}, S(P) \rangle \xhookrightarrow{\beta_2} D^b(X)
\end{align*}
where $X$ is the sequential blowup of $\mathbb{P}^3$ along two rational curves from Proposition~\ref{prop:kuznetsovembed}.

We first prove that the composition 
\[\Phi = \beta\circ \mathbb{L}_{S(P)}|_{P^\perp} \circ \alpha^* \colon D^b(X) \rightarrow D^b(X)
\] is Fourier-Mukai and that $\Phi(\langle P \rangle) \simeq 0$ where we note that $\alpha_1^* \simeq \mathbb{L}_P$. By corollary~\ref{cor:leftmutfm} and the fact that compositions of Fourier-Mukai functors are Fourier-Mukai, the functor $\Phi \colon D^b(X) \rightarrow D^b(X)$ is of Fourier-Mukai type. In particular, it is clear that $\Phi(\langle P \rangle) \simeq 0$ as this subcategory is annihilated by the functor $\mathbb{L}_P$. 

We now prove that the setting of Lemma~\ref{lem:extension} indeed satisfies the two additional assumptions of Proposition~\ref{prop:fmequiv}. As a consequence, the functor $\Phi$ will extend to an equivalence $\langle P^\perp, P \rangle \simeq \langle \prescript{\perp}{}P, S^{-1}P[1] \rangle$. To see that $\Phi|_{P^\perp}$ is an equivalence onto $\prescript{\perp}{}{P}$, we simply note that by \cite{Kuznetsov_2019}, the left mutation induces an equivalence $\mathbb{L}_{S(P)} \colon P^\perp \simeq  \prescript{\perp}{}{P}$. 

For the second claim, we note that as $S$ is an auto-equivalence, $S^{-1}(P)[1]$ is an exceptional object. So, it suffices to prove that there exists an isomorphism $\Phi(\alpha \alpha^!(P)) \simeq \beta \beta^!(S^{-1}P[1])$. To see this, observe that we have the following exact triangles:
\[
\begin{tikzcd}[row sep = tiny]
E\ \arrow[r] & P \arrow[r] & \tilde{P}[2]\\
S^{-1}E[1] \arrow[r] & \tilde{P}[-1] \arrow[r]& P[1]
\end{tikzcd}
\]
where the second line is obtained from the first by applying the functor $S^{-1}[1]$.
By uniqueness of the semi-orthogonal decomposition, we have that $\alpha_1\alpha_1^!(P) \simeq E$ and $\beta_1\beta_1^!(\tilde{P}[-1]) \simeq S^{-1}E[1]$. Then, from Lemma~\ref{lem:sinverse}, we have $\Phi(E) = \mathbb{L}_{S(P)}(E) \simeq S^{-1}(E)[1]$ and we conclude.
\end{proof}
\subsection{Extension of the natural isomorphism}\label{sec:extensioniso}
Let $\mathcal{I} \subset D^b(Q)$ be the full subcategory on the injective objects $I_1,I_2,I_3$ from Lemma~\ref{lem:projinj}. In this section, we prove the following:
\begin{lemma}\label{lem:extensioniso}
Assume $F_1, F_2 \colon D^b(Q) \rightarrow D^b(Q)$ equivalences of categories such that there exists a natural isomorphism $F_1|_{\mathcal{I}} \simeq F_2|_{\mathcal{I}}$ on the restrictions. Then there exists a natural isomorphism $F_1 \simeq F_2$.
\end{lemma}
This extension result relies on the fact that the category $D^b(Q)$ is \textbf{D}-standard which is the algebraic incarnation of the statement that any equivalence $D^b(B-mod) \rightarrow D^b(B-mod)$ for $B$ a finite dimensional $k$-algebra is of Fourier-Mukai type, i.e. representable by a bi-module.
\begin{defn}{\cite[Definition 5.1]{chen2017dstandard}}
A $k$-linear abelian category $\mathcal{A}$ is \textbf{D}\textit{-standard} over $k$ if the following holds: for any $k$-linear triangle autoequivalence $(F,\omega)\colon D^b(\mathcal{A}) \rightarrow D^b(\mathcal{A})$ satisfying $F(\mathcal{A}) \subseteq \mathcal{A}$ and any natural isomorphism $\theta_0 \colon F|_{\mathcal{A}} \rightarrow Id_{\mathcal{A}}$, there exists a natural isomorphism $\theta \colon (F,\omega) \rightarrow Id_{D^b(\mathcal{A})}$ of triangle functors extending $\theta_0$. Such a pair $(F,\omega)$ is called a \textit{pseudo-identity}.
\end{defn}
We introduce a particular class of categories which are \textbf{D}-standard and prove that our category $D^b(Q)$ does in fact belong to this class.
\begin{defn}
Let $A$ be a finite dimensional $k$-algebra. Let $\{S_1, \ldots, S_n\}$ be a complete set of representatives of simple modules over $A$. The \textit{Ext-quiver} $Q_A$ has vertex set $\{1 \ldots n\}$ and there is a unique arrow $i \rightarrow j$ if $Ext^1_A(S_i,S_j) \neq 0$. Moreover, the algebra $A$ is \textit{triangular} if the extension quiver $Q_A$ has no oriented cycles.
\end{defn}
As a result of \cite{chen2015note}, such algebras naturally yield \textbf{D}\text{-standard} categories.
\begin{prop}{\cite[Proposition 5.13]{chen2017dstandard}}\label{prop:triangularstandard}
Let $A$ be a finite dimensional $k$-algebra. If $A$ is triangular, then the abelian category of left modules $A-mod$ is \bf{D}\it{-standard}.
\end{prop}
\begin{lemma}\label{lem:bondaltriangular}
The path algebra $kQ/I$ associated with the Bondal quiver is a triangular algebra.
\end{lemma}
\begin{proof}
By definition, it suffices to verify that the extension quiver generated by the simple objects has no oriented cycles. The simple modules $S_i$ for $i = 1,2,3$ over the path algebra $kQ/I$ are given by the vertices of the quiver $Q$ and we claim that they satisfy the following:
\begin{itemize}
\item
$Ext^1(S_i,S_1) = 0$ for all $i$.
\item
$Ext^1(S_3,S_i) = 0$ for all $i$.
\item
$Ext^1(S_2,S_2) = 0$.
\end{itemize}
and in fact yield the following extension quiver
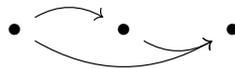
\[
\begin{tikzcd}
\bullet \arrow[r,bend left] \arrow[rr, bend right=30] & \bullet \arrow[r,bend right] & \bullet
\end{tikzcd}
\]
Indeed, the first two claims follow immediately as $S_1$ is injective and $S_3$ is projective. The third claim follows from the projective resolution of $S_2$:
\[
\begin{tikzcd}
P_3^{\oplus 2} \arrow[r] & P_2 \arrow[r] & S_2
\end{tikzcd}
\]
where the corresponding morphism on quiver representations are
\[
\begin{tikzcd}
0 \arrow[r,shift left]\arrow[r,shift right]& 0 \arrow[r,shift left]\arrow[r,shift right]& \mathbb{C}^2 \arrow[r] &0  \arrow[r,shift left]\arrow[r,shift right]& \mathbb{C}  \arrow[r,shift left]\arrow[r,shift right]& \mathbb{C}^2 \arrow[r] & 0  \arrow[r,shift left]\arrow[r,shift right]&  \mathbb{C}  \arrow[r,shift left]\arrow[r,shift right]& 0
\end{tikzcd}
\]
Taking the long exact sequence associated with the functor $Hom(\cdot, S_2)$, we find the exact sequence
\[
\begin{tikzcd}[column sep = 0.5cm]
Hom(S_2, S_2) \arrow[r] & Hom(P_2, S_2) \arrow[r] & Hom(P_3^{\oplus 2},S_2) \arrow[r] & Hom^1(S_2,S_2) \arrow[r] & Hom^1(P_2,S_2) \arrow[r] & \ldots
\end{tikzcd}
\]
But clearly, $Hom(P_3,S_2) = 0$, and as $P_2$ is projective, we have $Hom^1(P_2,S_2) = 0$. Thus, we have that $Hom^1(S_2,S_2) = 0$. The conclusion of the Lemma then follows immediately.
\end{proof}
We recall the standard fact that a natural isomorphism of two functors on the injective objects of an abelian category extends to the full abelian category.
\begin{lemma}\label{lem:huybrechtsextend}
Let $\mathcal{A}$ be a $k$-linear abelian category and let $F,G \colon D^b(\mathcal{A}) \rightarrow D^b(\mathcal{A})$ be exact autoequivalences and $\mathcal{I} \subset \mathcal{A}$ the full subcategory on a complete set of representatives of injective objects. Assume that any object in $D^b(A)$ admits a finite injective resolution. Assume that there exists a natural isomorphism $F|_\mathcal{I} \simeq G|_\mathcal{I} $. Then there exists a natural isomorphism of functors $F|_\mathcal{A} \simeq G|_\mathcal{A}$.
\end{lemma}
\begin{proof}
The proof of this is standard and follows, for example, from steps $1,2$ and $3$ in the proof of \cite[Proposition 4.23]{huybrechts2006fourier}.
\end{proof}
The main obstruction to applying the full statement of \cite[Proposition 4.23]{huybrechts2006fourier} is property (iii) in the definition of an ample sequence~\cite[Definition 4.19]{huybrechts2006fourier}, and the main contribution of this section is to bypass the last step of the argument by applying Proposition~\ref{prop:triangularstandard} in our example.
\begin{proof}[Proof of Lemma~\ref{lem:extensioniso}]
Let $A$ denote the path algebra with relations associated to the Bondal quiver. We prove that $F_1 \circ F_2^{-1} \simeq id_{D^b(Q)}$. First, it is clear that $A \simeq A^{op}$ as reversing the arrows and relations of \eqref{eq:bondal} yields an identical quiver with relations. Thus, left modules over $A$ are equivalent to right modules over $A$ and applying Lemma~\ref{lem:bondaltriangular} and Proposition~\ref{prop:triangularstandard}, it suffices to prove that the composition $F_1 \circ F_2^{-1}$ is a pseudo-identity. But by assumption, we have that $F_1 \circ F_2^{-1} |_{\mathcal{I}} \simeq id_{\mathcal{I}}$ and so by Lemma~\ref{lem:huybrechtsextend}, we have that $F_1 \circ F_2^{-1} |_{\mathcal{A}} \simeq id_{\mathcal{A}}$. In particular, we have that $F_1 \circ F_2^{-1}$ is a pseudo-identity and we conclude.
\end{proof}

\subsection{Reduction step to the category $P^\perp$}\label{sec:injectives}
In this section, we prove the following Lemma reducing the claim of a natural isomorphism to a calculation on the category $P^\perp$:
\begin{lemma}\label{lem:reduction}
Assume that there exists isomorphisms $\eta_{P^\perp} \colon \Phi\mathbb{L}_P I_i \simeq S^{-1}\mathbb{L}_P I_i[1]$ for all $i,j$ such that the following diagram commutes for any morphism $f \colon I_i \rightarrow I_j$.
\[
\begin{tikzcd}
\Phi(\mathbb{L}_PI_i) \arrow[r,"\Phi \mathbb{L}_P f"]\isoarrow{d}& \Phi(\mathbb{L}_PI_j) \isoarrow{d}\\
S^{-1}(\mathbb{L}_PI_i)[1] \arrow[r,"S^{-1}{[1]}(\mathbb{L}_Pf)"] &S^{-1}(\mathbb{L}_PI_j)[1]
\end{tikzcd}
\]
Then there exists isomorphisms $\eta_i \colon \Phi I_i \simeq S^{-1}I_i[1]$ for all $i,j$ such that the following diagram commutes for any morphism $f \colon I_i \rightarrow I_j$.
\[
\begin{tikzcd}
\Phi(I_i) \arrow[r,"\Phi(f)"]\isoarrow{d}& \Phi(I_j) \isoarrow{d}\\
S^{-1}(I_i)[1] \arrow[r,"S^{-1}f{[1]}"] &S^{-1}(I_j)[1]
\end{tikzcd}
\]
\end{lemma}
The proof of this Lemma relies on certain properties of the decomposition of the injective objects $I_i$ with respect to the semi-orthogonal decomposition $\langle P^\perp, P \rangle$. We recall the following Lemma.
\begin{lemma}\label{lem:exact}
Assume $\mathcal{T} = \langle \mathcal{B}, \mathcal{A} \rangle$ a $k$-linear triangulated category admitting a semi-orthogonal decomposition. Let $E_1,E_2 \in \mathcal{T}$ be objects with semi-orthogonal components given by $A_1,B_1$ and $A_2,B_2$ respectively. Then there exists an exact sequence:
\[
\begin{tikzcd}
Hom(B_1,A_2) \arrow[r] & Hom(E_1,E_2) \arrow[r] & Hom(A_1,A_2) \oplus Hom(B_1,B_2) \arrow[r] & Hom(B_1,A_2[1])
\end{tikzcd}
\]
\end{lemma}
\begin{proof}
This follows directly from the proof of \cite[Lemma 2.7]{2006math8430H} applied to the morphism of triangles
\[
\begin{tikzcd}
A_1 \arrow[r] \arrow[d] & E_1 \arrow[r]\arrow[d] & B_1\arrow[d] \\
A_2 \arrow[r] & E_2 \arrow[r] &B_2
\end{tikzcd}
\]
\end{proof}

We will use the following notation in the rest of this section. Recall that for each pair of diagrams in Lemma~\ref{lem:injective} and any morphism $f \in Hom(I_i,I_j)$, there exists unique morphisms $f_P, \mathbb{L}_Pf$ making the following diagram commute. 
\[
\begin{tikzcd}
P \arrow[d,dotted,"f_P"]\arrow[r] &I_i \arrow[r]\arrow[d, "f"] &\mathbb{L}_PI_i\arrow[d, dotted,"\mathbb{L}_Pf"] \arrow[r,"a_i"] & P[1] \arrow[d,"f_p{[1]}",dotted]\\
P \arrow[r] &I_j \arrow[r] &\mathbb{L}_PI_j \arrow[r,"a_j"] & P[1]\\
\end{tikzcd}
\]
We first establish the following technical Lemmas.
\begin{lemma}\label{lem:technical}
For any $i \geq j$, there exists a morphism $f \colon I_i \rightarrow I_j$ such that the induced map $f_P \colon P \rightarrow P$ is nonzero.
\end{lemma}
\begin{proof}
This is obvious for $i = j$ by functoriality of the semi-orthogonal decomposition by taking $f = id$. For $i = 3, j=2$ and $i=2, j=1$, we simply take the morphisms
\[
\begin{tikzcd}
\mathbb{C}^2 \arrow[r, shift left] \arrow[r, shift right]\arrow[d, "{\begin{psmallmatrix} 1&0\\0&0\end{psmallmatrix}}"]& \mathbb{C}^{2}\arrow[r, shift left] \arrow[r, shift right]\arrow[d, "{\begin{psmallmatrix} 1& 0\end{psmallmatrix}}"]& \mathbb{C}\arrow[d] & \mathbb{C}^2 \arrow[r, shift left] \arrow[r, shift right]\arrow[d, "{\begin{psmallmatrix} 1&0\end{psmallmatrix}}"]& \mathbb{C}\arrow[r, shift left] \arrow[r, shift right]\arrow[d]& 0\arrow[d] \\
\mathbb{C}^2 \arrow[r, shift left] \arrow[r, shift right]& \mathbb{C}\arrow[r, shift left] \arrow[r, shift right]& 0 &   \mathbb{C} \arrow[r, shift left] \arrow[r, shift right]& 0\arrow[r, shift left] \arrow[r, shift right]& 0 
\end{tikzcd}
\]
It is clear that for each of the above cases, there exists a morphism $g_i \colon P \rightarrow I_i, g_j \colon P \rightarrow I_j$ such that $f \circ g_i= g_j$. By uniqueness of the semi-orthogonal decomposition, the induced map $f_P$ must be the identity in each of the above cases. The case for $i=3,j=1$ follows immediately by composing the morphisms in the above two cases.
\end{proof}
In the following, we fix an isomorphism $\eta_P\colon\Phi(P) \simeq S^{-1}P[1]$ which exists by Proposition~\ref{prop:fmequiv}(2). 
\begin{lemma}\label{lem:extendiso}
Assume that there exists isomorphisms
\[
\eta_{P_i} \colon \Phi\mathbb{L}_P I_i \simeq S^{-1}\mathbb{L}_P I_i[1]
\]
such that for any pair of injective objects $I_i,I_j$ and any morphism $f \colon I_i \rightarrow I_j$, the left square in the following diagram commutes. Then up to a common rescaling $\eta_{P_i} \mapsto \alpha \eta_{P_i}$ of these isomorphisms, all squares in the following diagram commute for any $f$.
\begin{equation}\label{eq:twonatural}
\begin{tikzcd}[row sep=scriptsize, column sep=scriptsize]
& \Phi(\mathbb{L}_PI_i[-1])\arrow[dl,"\Phi(\mathbb{L}_Pf{[-1]})"'] \arrow[rr,"\Phi(a_i)"] \arrow[dd,"\eta_{P_i}" near start ] & & \Phi(P)\arrow[dl,"\Phi(f_P)"] \arrow[dd,"\eta_P"]\\
\Phi(\mathbb{L}_PI_j[-1])\arrow[rr, crossing over,"\Phi(a_j)" near end] \arrow[dd,"\eta_{P_j}"] & & \Phi(P) \arrow[dd,"\eta_P", near start] \\
& S^{-1}\mathbb{L}_PI_i \arrow[dl,"S^{-1}\mathbb{L}_Pf"' near start] \arrow[rr,"S^{-1}(a_i)" near start] & & S^{-1}P[1] \arrow[dl,"S^{-1}{[1]}(f_P)"] \\
S^{-1}\mathbb{L}_PI_j \arrow[rr,"S^{-1}(a_j)"] & & S^{-1}P[1] \arrow[from=uu, crossing over]  \\
\end{tikzcd}
\end{equation}
\end{lemma}
\begin{proof}
We first claim that for any common scaling $\eta_{P_i} \mapsto \alpha \eta_{P_i}$ of the isomorphisms in the assumption, the left, right, top, and bottom squares all commute. Indeed, the top and bottom squares commute by functoriality. The left square commutes by assumption. For the right square, recall that $Hom(P,P) = \mathbb{C}$. So $f_P = \beta\, id$, for some $\beta \in \mathbb{C}$. The square then obviously commutes as $\Phi(id) = S^{-1}[1](id) = id$ and by linearity.

It suffices to prove that there exists a common scalar $\alpha \in \mathbb{C}^*$, such that the isomorphisms $\alpha\eta_{P_j}$ make the following diagrams commute for all $j$.
\begin{equation}\label{eq:onenatural}
\begin{tikzcd}
\Phi(\mathbb{L}_PI_j[-1]) \arrow[r,"\Phi(a_j)"]\arrow[d,"\alpha\eta_{P_j}"]& \Phi(P) \arrow[d,"\eta_{P}"]\\
S^{-1}(\mathbb{L}_PI_j) \arrow[r,"S^{-1}(a_j)"] &S^{-1}(P)[1]
\end{tikzcd}
\end{equation}

We first prove the claim for a fixed $j$. We may assume that $S^{-1}(a_j) \circ \eta_{P_j} = g$ for some $g \neq 0$, as $\eta_{P_j}$ is an isomorphism and $a_j\neq 0$. On the other hand, we have that $Hom(\Phi(\mathbb{L}_PI_j[-1]),S^{-1}P[1]) = Hom(\Phi(\mathbb{L}_PI_j[-1]), \Phi(P)) = Hom(\mathbb{L}_PI_j[-1],P)= \mathbb{C}$ where the first equality follows from the isomorphism $S^{-1}P[1] \simeq \Phi(P)$, the second follows as $\Phi$ is an equivalence, and the third follows from Lemma~\ref{lem:extensions}(1,2,3). So $g = \alpha \eta_P \circ \Phi(a_j)$ for some nonzero $\alpha \in \mathbb{C}$ as $\eta_P$ is an isomorphism and $a_j$ is nontrivial. After making the redefinitions $\eta_{P_j} \mapsto \alpha \eta_{P_j}$ for all $j$, we see that the above diagram commutes and all previous diagrams in diagram~\eqref{eq:twonatural} still commute for any $f$.

We now prove that if the claim holds for the front square in diagram~\eqref{eq:twonatural}, then the claim holds for both the front and back squares. In the following, we redefine $\eta_{P_j}' = \alpha\eta_{P_j}, \eta_{P_i}' = \alpha\eta_{P_i}$ where $\alpha$ is obtained by applying the argument of the above paragraph to the front square. We first observe by Lemma~\ref{lem:technical} that there always exists $f$ such that $f_P \neq 0$ in which case it must be an isomorphism. Then, we have $S^{-1}[1](f_P) \circ S^{-1}(a_i) \circ \eta_{P_i}' = S^{-1}(a_j) \circ S^{-1}\mathbb{L}_P f \circ \eta_{P_i}'  = S^{-1}(a_j) \circ \eta_{P_j}' \circ \Phi(\mathbb{L}_Pf[-1])  = \eta_P \circ \Phi(a_j)\circ \Phi(\mathbb{L}_P f[-1])= \eta_P \circ \Phi(f_P) \circ \Phi(a_i) = S^{-1}[1](f_P) \circ \eta_P \circ \Phi(a_i) $. As $S^{-1}[1](f_P)$ is an isomorphism, we have that $S^{-1}(a_i)\circ \eta_{P_i}' = \eta_P\circ\Phi(a_i)$ and we conclude.

Applying the argument in the above paragraph in the cases $i= 2, j=1$, and $i = 3, j=2$, the claim follows.
\end{proof}
We now turn to proving Lemma~\ref{lem:reduction}.
\begin{proof}[Proof of Lemma~\ref{lem:reduction}]
We consider the following diagram:
\[
\begin{tikzcd}[row sep=scriptsize, column sep=scriptsize]
& \Phi(\mathbb{L}_PI_i)[-1]\arrow[dl,"\Phi(f)"'] \arrow[rr,"\Phi(a_i)"] \arrow[dd,"\eta_{P_i}" near start ] \arrow[dddl]& & \Phi(P)\arrow[dl,"\Phi(f_P)"'] \arrow[dd,"\eta_P", near start] \arrow[rr]\arrow[dddl]& & \Phi(I_i)\arrow[dl,"\Phi(f)"']\arrow[dd,"\eta_i",dotted]\\
\Phi(\mathbb{L}_PI_j)[-1]\arrow[rr, crossing over,"\Phi(a_j)" near end] \arrow[dd,"\eta_{P_j}"] & & \Phi(P) \arrow[dd,"\eta_P", near start] \arrow[rr, crossing over] & & \Phi(I_j) \\
& S^{-1}\mathbb{L}_PI_i \arrow[dl,"S^{-1}f"' near start] \arrow[rr,"S^{-1}(a_i)" near start] & & S^{-1}P[1] \arrow[dl,"S^{-1}{[1]}(f_P)"] \arrow[rr] & & S^{-1}I_i[1]\arrow[dl,"S^{-1}{[1]}(f)"]\\
S^{-1}\mathbb{L}_PI_j \arrow[rr,"S^{-1}(a_j)"] & & S^{-1}P[1] \arrow[from=uu, crossing over] \arrow[rr, crossing over] & & S^{-1}I_j[1] \arrow[from=uu, crossing over,"\eta_j" near start,dotted]\arrow[from=uuur, crossing over,dotted]\\
\end{tikzcd}
\]
We first prove that all squares other than the rightmost square in the diagram commute. By Lemma~\ref{lem:extendiso}, all squares in the left cube commute. The top and bottom squares on the right cube commute by functoriality. It suffices to prove that the front and back squares on the right cube commute.

To prove the existence of the morphisms $\eta_j,\eta_i$, we apply Lemma~\ref{lem:exact} to the front and back morphisms of distinguished triangles. For fixed $i$, this implies the exact sequence
\[
\begin{tikzcd}[row sep=scriptsize, column sep=scriptsize]
Hom(\Phi(\mathbb{L}_PI_i), S^{-1}P[1]) \arrow[r] & Hom(\Phi(I_i),S^{-1}I_i[1]) \arrow[ld,"\alpha"']\\ Hom(\Phi(P),S^{-1}P[1]) \oplus Hom(\Phi(\mathbb{L}_PI_i),S^{-1}\mathbb{L}_PI_i[1]) \arrow[r,"\beta"] & Hom(\Phi(\mathbb{L}_PI_i),S^{-1}P[2])
\end{tikzcd}
\]
By Lemma~\ref{lem:extensions}(1,2,3), $Hom(\Phi(\mathbb{L}_PI_i),S^{-1}P[1]) = Hom(\mathbb{L}_PI_i, P) = 0$ for all $i$. Thus, the map $\alpha$ in the above exact sequence is injective. On the other hand, $\beta(\eta_{P_i}[1],\eta_P) = 0$ by assumption and so there exists unique morphisms $\eta_i$ making the front and back diagrams commute on the right cube. In particular, these must be isomorphisms as $\eta_P,\eta_{P_i}$ are both isomorphisms and the front and back are morphisms of distinguished triangles.

Finally, to prove that the rightmost square commutes, we apply Lemma~\ref{lem:exact} to the diagonal morphism of distinguished triangles. By the same argument as in the above paragraph, there exists a unique diagonal morphism on the right square making the diagonal diagram commute. On the other hand, as all other squares commute, the two compositions $\eta_j \circ \Phi(f), S^{-1}[1](f) \circ \eta_i$ also make the diagonal digram commute. Since this must be unique, we conclude that $\eta_j \circ \Phi(f)  =S^{-1}[1](f) \circ \eta_i$.
\end{proof}
\subsection{Verification of compatibility on $P^\perp$}\label{sec:compatibility}
In this section, we prove that the assumption of Lemma~\ref{lem:reduction} holds via a direct computation. Specifically, we will prove the following:
\begin{lemma}\label{lem:computation}
There exists a natural isomorphism of functors $\Phi \mathbb{L}_P|_{\mathcal{I}} \simeq S^{-1} [1]\mathbb{L}_P|_{\mathcal{I}}$ where $\mathcal{I} \subset D^b(Q)$ is the full subcategory on the three injective objects $\{ I_1, I_2, I_3 \} \subset \mathrm{mod}\text{-}kQ$. 
\end{lemma}
In Section~\ref{simplification}, we prove Lemma~\ref{lem:calculation} reducing Lemma~\ref{lem:computation} to a verification of compatibility of diagrams~\eqref{eq:first} and \eqref{eq:second}. In sections~\ref{verification1} and \ref{verification2}, we perform the explicit computations in the homotopy category of complexes to verify the compatibility.
\subsubsection{Simplification of computation}\label{simplification}
The goal of this section is to simplify the statement of Lemma~\ref{lem:computation}. Specifically, we prove the following.
{\setlength{\emergencystretch}{0.5\textwidth}\par}
\begin{lemma}\label{lem:calculation}
Recall the exact triangles from Lemmas~\ref{lem:injective} and \ref{lem:projective} corresponding to the semi-orthogonal decompositions $D^b(Q) = \langle P^\perp, P \rangle = \langle S(P), P^\perp \rangle$ respectively.
\begin{equation}\label{eq:first}
\begin{tikzcd}
P \arrow[r] \arrow[d]&I_3 \arrow[r]\arrow[d,"f"]& \tilde{C} \arrow[d,"f'"] & \tilde{C} \arrow[r] \arrow[d,"f'"]& P_3[1] \arrow[r]\arrow[d,"S^{-1}f{[1]}"] & \tilde{P}[1]\arrow[d] \\
P \arrow[r]& I_2 \arrow[r]&A & A \arrow[r] & P_2[1] \arrow[r] & \tilde{P}[1] 
\end{tikzcd}
\end{equation}
\begin{equation}\label{eq:second}
\begin{tikzcd}
P \arrow[r]\arrow[d]& I_2 \arrow[r]\arrow[d,"g"]&A \arrow[d,"g'"] & A \arrow[r] \arrow[d,"g'"]& P_2[1] \arrow[r]\arrow[d,"S^{-1}g{[1]}"] & \tilde{P}[1] \arrow[d]\\
P \arrow[r] & I_1 \arrow[r] & D[1] & D[1] \arrow[r] & P_1[1] \arrow[r] & \tilde{P}[1]
\end{tikzcd}
\end{equation}
Assume that for each element of a basis $f \colon I_3 \rightarrow I_2$, there exists an $f'$ making both of the top two diagrams commute. Similarly, assume that for each element of a basis $g \colon I_2 \rightarrow I_1$, there exists a $g'$ making the both of the bottom two diagrams commute. Then there exists a natural isomorphism of functors $\Phi \mathbb{L}_P|_{\mathcal{I}} \simeq S^{-1} [1]\mathbb{L}_P|_{\mathcal{I}}$. 
\end{lemma}
We begin with the following observation.
\begin{lemma}\label{lem:calculation0reduction}
Consider the full subcategory $\mathcal{I}$ on the injective objects $\{I_1,I_2,I_3 \}$ and $\mathcal{P}$ on the projective objects $\{P_1,P_2,P_3\}$. Assume that there exists an isomorphism of functors $\mathbb{L}_P \simeq \mathbb{R}_{S(P)} S^{-1}[1]$ on $\mathcal{I}$. Then there exists an isomorphism of functors $\Phi\mathbb{L}_P |_{\mathcal{I}} \simeq S^{-1}[1] \mathbb{L}_P |_{\mathcal{I}}$.
\end{lemma}
\begin{proof}
By definition, there exists a natural isomorphism $\Phi\mathbb{L}_P \simeq \mathbb{L}_{S(P)}\mathbb{L}_P$. Thus, we have the following chain of isomorphisms for any morphism of injective objects $f \colon I_i \rightarrow I_j$:
\[
\begin{tikzcd}
S^{-1}\mathbb{L}_PI_i[1] \arrow[d] \arrow[r,"\sim"] & \mathbb{L}_{S(P)}S^{-1}I_i[1] \arrow[d] \arrow[r,"\sim"] & \mathbb{L}_{S(P)}\mathbb{R}_{S(P)}S^{-1}I_i[1] \arrow[r,"\sim"] \arrow[d]& \mathbb{L}_{S(P)}\mathbb{L}_P I_i \arrow[d]\\
S^{-1} \mathbb{L}_PI_j[1] \arrow[r,"\sim"] & \mathbb{L}_{S(P)}S^{-1}I_j[1] \arrow[r,"\sim"] & \mathbb{L}_{S(P)}\mathbb{R}_{S(P)}S^{-1}I_j[1] \arrow[r,"\sim"] & \mathbb{L}_{S(P)}\mathbb{L}_PI_j
\end{tikzcd}
\]
where the first square follows from the standard isomorphism of functors~\cite[Lemma 2.3(ii)]{Kuznetsov_2019} applied to the case $S \circ \mathbb{L}_P \circ S^{-1} \simeq \mathbb{L}_{S(P)}$, the middle square of isomorphisms follows by applying the functor $\mathbb{L}_{S(P)}$ to the following diagram by Lemmas~\ref{lem:serre} and \ref{lem:projective}:
\[
\begin{tikzcd}
\mathbb{R}_{S(P)}P_i[1] \arrow[r] \arrow[d] & P_i[1] \arrow[r]\arrow[d] & \tilde{P}[1]\arrow[d] \\
\mathbb{R}_{S(P)}P_j[1] \arrow[r]& P_j[1] \arrow[r]& \tilde{P}[1]
\end{tikzcd}
\]
and the last square of isomorphisms follows by the assumption.
\end{proof}
In particular, it suffices to verify the isomorphism $\mathbb{L}_P |_{\mathcal{I}} \simeq \mathbb{R}_{S(P)}S^{-1}[1] |_{\mathcal{I}}$ on a basis of morphisms between distinct injective objects.
\begin{lemma}\label{lem:simplify}
Let $f_1,f_2 \colon I_3 \rightarrow I_2$ and $g_1,g_2 \colon I_2 \rightarrow I_1$ be a basis of morphisms. Assume that there exists isomorphisms $\eta_{P^\perp} \colon \mathbb{L}_P I_i \simeq \mathbb{R}_{S(P)}S^{-1} I_i[1]$ for all $i$ such that the following diagram commutes with $f$ any of the above four morphisms.
\[
\begin{tikzcd}
\mathbb{L}_PI_i \arrow[r," \mathbb{L}_P f"]\isoarrow{d}& \mathbb{L}_PI_j \isoarrow{d}\\
\mathbb{R}_{S(P)}S^{-1}(I_i)[1] \arrow[r,"\mathbb{R}_{S(P)}S^{-1}(f){[1]}"] &\mathbb{R}_{S(P)}S^{-1}(I_j)[1]
\end{tikzcd}
\]
Then the diagram commutes for any pair $i,j$ and any morphism $f \in Hom(I_i,I_j)$.
\end{lemma}
\begin{proof}
The claim is obvious if $i = j$. Indeed, any $f \in Hom(I_i,I_i) = \mathbb{C}$ is of the form $f = \alpha\cdot id$ as the injective objects $I_i$ are simple. The diagram then commutes as $\mathbb{L}_P(id) = \mathbb{R}_{S(P)}S^{-1}[1](id) = id$ and by linearity.

As the composition map $Hom(I_3,I_2) \times Hom(I_2,I_1) \rightarrow Hom(I_3,I_1)$ is surjective, the remaining claim again follows by linearity and additivity of the functors $\mathbb{L}_P, \mathbb{R}_{S(P)}S^{-1}[1]$ and of compositions.
\end{proof}

We now prove Lemma~\ref{lem:calculation}
\begin{proof}[Proof of Lemma~\ref{lem:calculation}] 
The claim essentially follows from the uniqueness of a semi-orthogonal decomposition. For the sequences in \eqref{eq:first}, we have the following sequence of natural isomorphisms
\[
\begin{tikzcd}
\mathbb{L}_PI_3 \arrow[r,"\sim"] \arrow[d,"\mathbb{L}_Pf"]& \tilde{C} \arrow[r,"\sim"] \arrow[d,"f'"]& \mathbb{R}_{S(P)}P_3[1] \arrow[r, "\sim"] \arrow[d] & \mathbb{R}_{S(P)}S^{-1}I_3[1]\arrow[d,"S^{-1}f{[1]}"]\\
\mathbb{L}_PI_2 \arrow[r,"\sim"] & A \arrow[r,"\sim"] & \mathbb{R}_{S(P)}P_2[1] \arrow[r,"\sim"]  & \mathbb{R}_{S(P)}S^{-1}I_2[1]
\end{tikzcd}
\]
where the first and second squares follow from the uniqueness of semi-orthogonal decompositions and the third follows from Lemma~\ref{lem:serre}. The conclusion for the sequences in ~\eqref{eq:second} follows from an identical argument.

Then by Lemma~\ref{lem:simplify}, we conclude that there exists a natural isomorphism of functors $\mathbb{L}_P \simeq \mathbb{R}_{S(P)}S^{-1}[1]$ on the full subcategory $\mathcal{I}$ of injective objects. The conclusion then follows from Lemma~\ref{lem:calculation0reduction}.
{\setlength{\emergencystretch}{0.5\textwidth}\par}
\end{proof}
\subsubsection{Verification of compatibility for sequences in~\eqref{eq:first}}\label{verification1}
In this section, we prove part 1 in the assumption of Lemma~\ref{lem:calculation}. Our calculation is particularly lengthy, in part, because we must explicitly verify the equality of the two morphisms by passing through cones and projective/injective resolutions.
\begin{lemma}\label{lem:calculation1}
For each element of a basis $f \colon I_3 \rightarrow I_2$, there exists an $f'$ making the following two diagrams commute.
\[
\begin{tikzcd}
P \arrow[r] \arrow[d]&I_3 \arrow[r,"i_3"]\arrow[d,"f"]& \tilde{C} \arrow[d,"f'"] & \tilde{C} \arrow[r,"p_3"] \arrow[d,"f'"]& P_3[1] \arrow[r]\arrow[d,"S^{-1}f{[1]}"] & \tilde{P}[1]\arrow[d] \\
P \arrow[r]& I_2 \arrow[r,"i_2"]&A & A \arrow[r,"p_2"] & P_2[1] \arrow[r] & \tilde{P}[1] 
\end{tikzcd}
\]
\end{lemma}
We first summarize the steps, illustrated in figure~\ref{fig:1}, describing the explicit computations involved in the homotopy category of complexes. It suffices to check the following for a basis of the vector space $Hom(I_3,I_2)$ given by the morphisms described in Lemma~\ref{lem:serre}.
\begin{enumerate}[(a)]
\item
In steps $1$ and $2$, we compute the composition of the morphism $i_2 \circ f \colon I_3 \rightarrow A$ with the injective resolution $A \rightarrow I_A^\bullet$ described in Lemma~\ref{lem:aresolution}.
\item
In step $3$, we compute the unique induced morphism $\tilde{C} \rightarrow I_A^\bullet$ whose composition with $i_3$ agrees with the results of (a).
\item
In steps $4$ and $5$, we compute the composition of the morphism $S^{-1}f[1] \circ p_3$ with the projective resolution $P_{\tilde{C}}^\bullet \xrightarrow{\sim} \tilde{C}$ of Lemma~\ref{lem:cdresolution}.
\item
In steps $6,7$ and $8$, we compute the composition $P_{\tilde{C}}^\bullet \xrightarrow{f'} A \xrightarrow{\sim} Cone(P_2 \rightarrow \tilde{P})$ with $f'$ obtained from part (b), where we recall that the morphism $p_2$ factors into the composition $A \xrightarrow{\sim} Cone(P_2\rightarrow \tilde{P}) \rightarrow P_2[1]$ by definition.
\item
In steps $9$ and $10$, we use the above step to deduce the composition of the morphism $p_2 \circ f'$ with the projective resolution $P_{\tilde{C}}^\bullet \xrightarrow{\sim} \tilde{C}$ and check that this agrees with the result from part (c).
\end{enumerate}
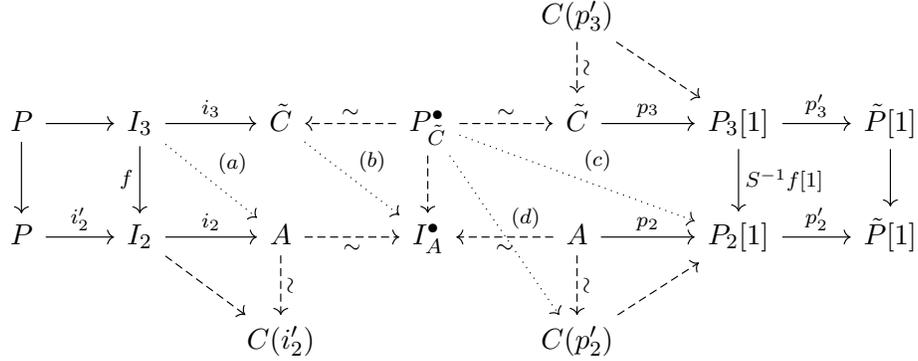
\begin{figure}[H]
\begin{tikzcd}
& & & &C(p_3') \isoarrow{d,dashed} \arrow[rd,dashed] & &\\
P \arrow[r] \arrow[d]&I_3 \arrow[r,"i_3"]\arrow[d,"f"']\arrow[rd,dotted,"(a)"]& \tilde{C}  \arrow[rd,dotted,"(b)"] &[-5pt]P_{\tilde{C}}^\bullet \arrow[rrd,dotted,"(c)"] \arrow[ rdd,dotted,"(d)"]\arrow[l,dashed,"\sim"'] \arrow[r,dashed,"\sim"] \arrow[d,dashed]&[-5pt]\tilde{C} \arrow[r,"p_3"]& P_3[1] \arrow[r,"p_3'"]\arrow[d,"S^{-1}f{[1]}"] & \tilde{P}[1]\arrow[d] \\
P \arrow[r,"i_2'"]& I_2 \arrow[r,"i_2"] \arrow[rd,dashed]&A  \arrow[r,dashed,"\sim"'] \isoarrow{d,dashed}&  I_{A}^\bullet & A \isoarrow{d,dashed} \arrow[l,dashed,"\sim"] \arrow[r,"p_2"] & P_2[1] \arrow[r,"p_2'"] & \tilde{P}[1] \\
& & C(i_2')& &C(p_2') \arrow[ru,dashed] & & 
\end{tikzcd}
\caption{Illustration of the steps involved in the verification. In the diagram, $C(f)$ denotes the cone of a morphism $f$, the dashed arrows denote the canonical morphisms, the dashed isomorphisms come from Lemma~\ref{lem:aresolution},~\ref{lem:cdresolution}, and the dotted arrows correspond to the indicated steps.}
\label{fig:1}
\end{figure}
In more detail, the proof of Lemma~\ref{lem:calculation1} is broken down into the following steps for each map $f \colon I_3 \rightarrow I_2$.
\begin{lemma}[Step 1]
The composition $Cone(P \rightarrow I_2) \xleftarrow{\sim} A \xrightarrow{\sim} I_A^\bullet$ is given by the following morphism of complexes.
\[
\begin{tikzcd}
\mathbb{C}\arrow[r, shift left, "1"] \arrow[r, shift right, "0", labels = below] \arrow[d,"{\begin{psmallmatrix}1 \\0 \end{psmallmatrix}}"] &\mathbb{C} \arrow[r, shift left, "1"] \arrow[r, shift right, "0", labels = below]\arrow[d, "{\begin{psmallmatrix}1\\0 \end{psmallmatrix}}"]& \mathbb{C} \arrow[d, "1"]
\arrow[r]& \mathbb{C}^2 \arrow[d, "{\begin{psmallmatrix}1&0\\0&1\\0&0\\0&0 \end{psmallmatrix}}"] \arrow[r, shift left] \arrow[r, shift right] &\mathbb{C} \arrow[d, "{\begin{psmallmatrix}1\\0 \end{psmallmatrix}}"] \arrow[r, shift left] \arrow[r, shift right]  & 0 \arrow[d]\\
\mathbb{C}^2 \arrow[r, shift left] \arrow[r, shift right]  & \mathbb{C}^2 \arrow[r, shift left] \arrow[r, shift right] & \mathbb{C} \arrow[r] & \mathbb{C}^4 \arrow[r, shift left] \arrow[r, shift right] & \mathbb{C}^2 \arrow[r, shift left] \arrow[r, shift right] & 0 \arrow[r] & \mathbb{C} \arrow[r, shift left] \arrow[r, shift right] & 0 \arrow[r, shift left] \arrow[r, shift right] & 0
\end{tikzcd}
\]
\end{lemma}
\begin{proof}
This comes from combining the quasi-isomorphism $A \xrightarrow{\sim} Cone(P\rightarrow I_2)$ described in Lemma~\ref{lem:injective} with Lemma~\ref{lem:aresolution}.
{\setlength{\emergencystretch}{0.5\textwidth}\par}
\end{proof}
\begin{lemma}[Step 2]
The composition 
 $I_3 \rightarrow I_2 \rightarrow Cone(P\rightarrow I_2) \xrightarrow{\sim} I_A^\bullet$ is given by the following morphism of complexes for each map $f_1, f_2 \colon I_3 \rightarrow I_2$ from Lemma~\ref{lem:serre}.
\[
\begin{tikzcd}
&&& \mathbb{C}^2 \arrow[d, "{\begin{psmallmatrix}1&0\\0&0\\0&0\\0&0 \end{psmallmatrix}}"] \arrow[r, shift left] \arrow[r, shift right] &\mathbb{C}^2 \arrow[d, "{\begin{psmallmatrix}1&0\\0&0 \end{psmallmatrix}}"] \arrow[r, shift left] \arrow[r, shift right]  & \mathbb{C}\arrow[d] \\
\mathbb{C}^2 \arrow[r, shift left] \arrow[r, shift right]  & \mathbb{C}^2 \arrow[r, shift left] \arrow[r, shift right] & \mathbb{C} \arrow[r] & \mathbb{C}^4 \arrow[r, shift left] \arrow[r, shift right] & \mathbb{C}^2 \arrow[r, shift left] \arrow[r, shift right] & 0 \arrow[r] & \mathbb{C} \arrow[r, shift left] \arrow[r, shift right] & 0 \arrow[r, shift left] \arrow[r, shift right] & 0
\end{tikzcd}
\]
\[
\begin{tikzcd}
&&& \mathbb{C}^2 \arrow[d, "{\begin{psmallmatrix}0&0\\0&1\\0&0\\0&0 \end{psmallmatrix}}"] \arrow[r, shift left] \arrow[r, shift right] &\mathbb{C}^2 \arrow[d, "{\begin{psmallmatrix}0&1\\0&0 \end{psmallmatrix}}"] \arrow[r, shift left] \arrow[r, shift right]  & \mathbb{C}\arrow[d] \\
\mathbb{C}^2 \arrow[r, shift left] \arrow[r, shift right]  & \mathbb{C}^2 \arrow[r, shift left] \arrow[r, shift right] & \mathbb{C} \arrow[r] & \mathbb{C}^4 \arrow[r, shift left] \arrow[r, shift right] & \mathbb{C}^2 \arrow[r, shift left] \arrow[r, shift right] & 0 \arrow[r] & \mathbb{C} \arrow[r, shift left] \arrow[r, shift right] & 0 \arrow[r, shift left] \arrow[r, shift right] & 0
\end{tikzcd}
\]
\end{lemma}
\begin{proof}
This comes from combining the obvious maps $I_3 \rightarrow I_2 \rightarrow Cone(P \rightarrow I_2)$ with the result of step 1.
\end{proof}
\begin{lemma}[Step 3]
The unique morphism
$\tilde{C} \rightarrow I_A^\bullet$ making the precomposition $I_3 \rightarrow \tilde{C} \rightarrow I_A^\bullet$ agree with step 2 is given by the following morphism of complexes.
\[
\begin{tikzcd}
&&& \mathbb{C} \arrow[d, "{\begin{psmallmatrix}0 \\0\\0\\-1 \end{psmallmatrix}}"'] \arrow[r, shift left,"0"] \arrow[r, shift right,"1"'] &\mathbb{C}\arrow[d, "{\begin{psmallmatrix}0\\-1 \end{psmallmatrix}}"] \arrow[r, shift left] \arrow[r, shift right]  & 0 \arrow[d]\\
\mathbb{C}^2 \arrow[r, shift left] \arrow[r, shift right]  & \mathbb{C}^2 \arrow[r, shift left] \arrow[r, shift right] & \mathbb{C} \arrow[r] & \mathbb{C}^4 \arrow[r, shift left] \arrow[r, shift right] & \mathbb{C}^2 \arrow[r, shift left] \arrow[r, shift right] & 0 \arrow[r] & \mathbb{C} \arrow[r, shift left] \arrow[r, shift right] & 0 \arrow[r, shift left] \arrow[r, shift right] & 0
\end{tikzcd}
\]
\[
\begin{tikzcd}
&&& \mathbb{C} \arrow[d, "{\begin{psmallmatrix}0\\1\\0\\0 \end{psmallmatrix}}"'] \arrow[r, shift left,"0"] \arrow[r, shift right,"1"'] &\mathbb{C}\arrow[d, "{\begin{psmallmatrix}1\\0 \end{psmallmatrix}}"] \arrow[r, shift left] \arrow[r, shift right]  & 0\arrow[d] \\
\mathbb{C}^2 \arrow[r, shift left] \arrow[r, shift right]  & \mathbb{C}^2 \arrow[r, shift left] \arrow[r, shift right] & \mathbb{C} \arrow[r] & \mathbb{C}^4 \arrow[r, shift left] \arrow[r, shift right] & \mathbb{C}^2 \arrow[r, shift left] \arrow[r, shift right] & 0 \arrow[r] & \mathbb{C} \arrow[r, shift left] \arrow[r, shift right] & 0 \arrow[r, shift left] \arrow[r, shift right] & 0
\end{tikzcd}
\]
\end{lemma}
\begin{proof}
The morphism $I_3 \rightarrow I_A^\bullet$ specified by the following
\[
\begin{tikzcd}
&&&\mathbb{C}^2 \arrow[d, "{\begin{psmallmatrix}1&0\\0&0\\0&0\\0&1 \end{psmallmatrix}}"] \arrow[r, shift left] \arrow[r, shift right] &\mathbb{C}^2 \arrow[d, "{\begin{psmallmatrix}1&0\\0&1 \end{psmallmatrix}}"] \arrow[r, shift left] \arrow[r, shift right]  & \mathbb{C}\arrow[d] \\
\mathbb{C}^2 \arrow[r, shift left] \arrow[r, shift right]  & \mathbb{C}^2 \arrow[r, shift left] \arrow[r, shift right] & \mathbb{C} \arrow[r] & \mathbb{C}^4 \arrow[r, shift left] \arrow[r, shift right] & \mathbb{C}^2 \arrow[r, shift left] \arrow[r, shift right] & 0 \arrow[r] & \mathbb{C} \arrow[r, shift left] \arrow[r, shift right] & 0 \arrow[r, shift left] \arrow[r, shift right] & 0
\end{tikzcd}
\]
is homotopy equivalent to the zero morphism. The result follows by checking that the composition of the above morphisms with the map $I_3 \rightarrow \tilde{C}$ agrees with step 2 up to homotopy equivalence.
\end{proof}
\begin{lemma}[Step 4]
The composition $P_{\tilde{C}}^\bullet \xrightarrow{\sim} \tilde{C} \xleftarrow{\sim} Cone(P_3 \rightarrow \tilde{P})$ is given by the following morphism of complexes.
\[
\begin{tikzcd}
0 \arrow[r, shift left] \arrow[r, shift right] & 0 \arrow[r, shift left] \arrow[r, shift right] & \mathbb{C}\arrow[r] &0\arrow[d] \arrow[r, shift left] \arrow[r, shift right] & \mathbb{C}\arrow[d] \arrow[r, shift left] \arrow[r, shift right] & \mathbb{C}^3 \arrow[r] \arrow[d,"{\begin{psmallmatrix}0\\0\\1 \end{psmallmatrix}}"]& \mathbb{C}\arrow[d,"1"] \arrow[r, shift left] \arrow[r, shift right]  & \mathbb{C}^2 \arrow[d,"{\begin{psmallmatrix}0&1\end{psmallmatrix}}"]\arrow[r, shift left] \arrow[r, shift right] & \mathbb{C}^2 \arrow[d,"{\begin{psmallmatrix}0&1\end{psmallmatrix}}"]\\
&&& 0 \arrow[r, shift left] \arrow[r, shift right] & 0 \arrow[r, shift left] \arrow[r, shift right] & \mathbb{C} \arrow[r] & \mathbb{C} \arrow[r, shift left,"0"] \arrow[r, shift right,"1"'] & \mathbb{C} \arrow[r, shift left,"0"] \arrow[r, shift right,"1"'] & \mathbb{C}
\end{tikzcd}
\]
\end{lemma}
\begin{proof}
This comes from combining the resolution of Lemma~\ref{lem:cdresolution} with the natural map $Cone(P_3 \rightarrow \tilde{P}) \xrightarrow{\sim} \tilde{C}$ of Lemma~\ref{lem:projective}.
\end{proof}
\begin{lemma}[Step 5]
The composition $P_{\tilde{C}}^\bullet \rightarrow Cone(P_3 \rightarrow \tilde{P}) \rightarrow P_3[1] \rightarrow P_2[1]$ is given by the following morphism of complexes.
\[
\begin{tikzcd}
0 \arrow[r, shift left] \arrow[r, shift right] & 0 \arrow[r, shift left] \arrow[r, shift right] & \mathbb{C}\arrow[r] &0\arrow[d] \arrow[r, shift left] \arrow[r, shift right] & \mathbb{C}\arrow[d,"0"] \arrow[r, shift left] \arrow[r, shift right] & \mathbb{C}^3\arrow[d,"{\begin{psmallmatrix}0&0&1\\0&0&0\end{psmallmatrix}}"] \arrow[r] & \mathbb{C} \arrow[r, shift left] \arrow[r, shift right]  & \mathbb{C}^2 \arrow[r, shift left] \arrow[r, shift right] & \mathbb{C}^2 \\
&&& 0 \arrow[r, shift left] \arrow[r, shift right] & \mathbb{C} \arrow[r, shift left] \arrow[r, shift right] & \mathbb{C}^2 
\end{tikzcd}
\]
\[
\begin{tikzcd}
0 \arrow[r, shift left] \arrow[r, shift right] & 0 \arrow[r, shift left] \arrow[r, shift right] & \mathbb{C}\arrow[r] &0\arrow[d] \arrow[r, shift left] \arrow[r, shift right] & \mathbb{C}\arrow[d,"0"] \arrow[r, shift left] \arrow[r, shift right] & \mathbb{C}^3\arrow[d,"{\begin{psmallmatrix}0&0&0\\0&0&1\end{psmallmatrix}}"] \arrow[r] & \mathbb{C} \arrow[r, shift left] \arrow[r, shift right]  & \mathbb{C}^2 \arrow[r, shift left] \arrow[r, shift right] & \mathbb{C}^2 \\
&&& 0 \arrow[r, shift left] \arrow[r, shift right] & \mathbb{C} \arrow[r, shift left] \arrow[r, shift right] & \mathbb{C}^2 
\end{tikzcd}
\]
\end{lemma}
\begin{proof}
This comes from combining step 4 with the obvious maps where the map $P_3[1] \rightarrow P_2[1]$ is induced from the map $f \colon I_3 \rightarrow I_2$ under the Serre functor described in Lemma~\ref{lem:serre}.
\end{proof}
\begin{lemma}[Step 6]
The composition $P_{\tilde{C}}^\bullet \rightarrow \tilde{C} \rightarrow I_A^\bullet$ is given by the following morphism of complexes.
\[
\begin{tikzcd}[column sep = 0.75cm]
0 \arrow[r, shift left] \arrow[r, shift right] & 0 \arrow[r, shift left] \arrow[r, shift right] & \mathbb{C}\arrow[r] &0\arrow[d] \arrow[r, shift left] \arrow[r, shift right] & \mathbb{C} \arrow[d,"0"]\arrow[r, shift left] \arrow[r, shift right] & \mathbb{C}^3 \arrow[d,"0"]\arrow[r] & \mathbb{C} \arrow[d,"{\begin{psmallmatrix}0\\0\\0\\-1 \end{psmallmatrix}}"'] \arrow[r, shift left] \arrow[r, shift right]  & \mathbb{C}^2 \arrow[d,"{\begin{psmallmatrix}0&0\\0 &-1 \end{psmallmatrix}}"']\arrow[r, shift left] \arrow[r, shift right] & \mathbb{C}^2 \arrow[d] \\
&&&\mathbb{C}^2 \arrow[r, shift left] \arrow[r, shift right]  & \mathbb{C}^2 \arrow[r, shift left] \arrow[r, shift right] & \mathbb{C} \arrow[r] & \mathbb{C}^4 \arrow[r, shift left] \arrow[r, shift right] & \mathbb{C}^2 \arrow[r, shift left] \arrow[r, shift right] & 0 \arrow[r] & \mathbb{C} \arrow[r, shift left] \arrow[r, shift right] & 0 \arrow[r, shift left] \arrow[r, shift right] & 0
\end{tikzcd}
\]
\[
\begin{tikzcd}[column sep = 0.75cm]
0 \arrow[r, shift left] \arrow[r, shift right] & 0 \arrow[r, shift left] \arrow[r, shift right] & \mathbb{C}\arrow[r] &0\arrow[d] \arrow[r, shift left] \arrow[r, shift right] & \mathbb{C} \arrow[d,"0"]\arrow[r, shift left] \arrow[r, shift right] & \mathbb{C}^3 \arrow[d,"0"]\arrow[r] & \mathbb{C} \arrow[d,"{\begin{psmallmatrix}0\\1\\0\\0 \end{psmallmatrix}}"'] \arrow[r, shift left] \arrow[r, shift right]  & \mathbb{C}^2 \arrow[d,"{\begin{psmallmatrix}0&1\\0&0 \end{psmallmatrix}}"']\arrow[r, shift left] \arrow[r, shift right] & \mathbb{C}^2 \arrow[d] \\
&&&\mathbb{C}^2 \arrow[r, shift left] \arrow[r, shift right]  & \mathbb{C}^2 \arrow[r, shift left] \arrow[r, shift right] & \mathbb{C} \arrow[r] & \mathbb{C}^4 \arrow[r, shift left] \arrow[r, shift right] & \mathbb{C}^2 \arrow[r, shift left] \arrow[r, shift right] & 0 \arrow[r] & \mathbb{C} \arrow[r, shift left] \arrow[r, shift right] & 0 \arrow[r, shift left] \arrow[r, shift right] & 0
\end{tikzcd}
\]
\end{lemma}
\begin{proof}
This comes from combining the result of step 3 with the projective resolution~\ref{lem:cdresolution}.
\end{proof}
\begin{lemma}[Step 7]
The composition $Cone(P_2 \rightarrow \tilde{P}) \xleftarrow{\sim} A \xrightarrow{\sim} I_A^\bullet$ is given by the following morphism of complexes.
\[
\begin{tikzcd}
0\arrow[d] \arrow[r, shift left] \arrow[r, shift right] & \mathbb{C}\arrow[d,"{\begin{psmallmatrix}1\\0\end{psmallmatrix}}"] \arrow[r, shift left] \arrow[r, shift right] & \mathbb{C}^2 \arrow[r] \arrow[d,"{\begin{psmallmatrix}1 & 0\end{psmallmatrix}}"]& \mathbb{C}\arrow[d,"{\begin{psmallmatrix}0\\1\\0\\0 \end{psmallmatrix}}"] \arrow[r, shift left,"0"] \arrow[r, shift right,"1"']  & \mathbb{C} \arrow[d,"{\begin{psmallmatrix}1\\0\end{psmallmatrix}}"]\arrow[r, shift left,"0"] \arrow[r, shift right,"1"'] & \mathbb{C} \arrow[d]\\
\mathbb{C}^2 \arrow[r, shift left] \arrow[r, shift right]  & \mathbb{C}^2 \arrow[r, shift left] \arrow[r, shift right] & \mathbb{C} \arrow[r] & \mathbb{C}^4 \arrow[r, shift left] \arrow[r, shift right] & \mathbb{C}^2 \arrow[r, shift left] \arrow[r, shift right] & 0 \arrow[r] & \mathbb{C} \arrow[r, shift left] \arrow[r, shift right] & 0 \arrow[r, shift left] \arrow[r, shift right] & 0
\end{tikzcd}
\]
\end{lemma}
\begin{proof}
This comes from combining the resolution~\ref{lem:aresolution} with the natural morphism $A \xrightarrow{\sim} Cone(P_2 \rightarrow \tilde{P})$ of Lemma~\ref{lem:projective}.
\end{proof}
\begin{lemma}[Step 8]
The composition $P_{\tilde{C}}^\bullet \rightarrow Cone(P_2 \rightarrow \tilde{P})$ is given by the following morphism of complexes.
\[
\begin{tikzcd}
0 \arrow[r, shift left] \arrow[r, shift right] & 0 \arrow[r, shift left] \arrow[r, shift right] & \mathbb{C}\arrow[r] &0\arrow[d] \arrow[r, shift left] \arrow[r, shift right] & \mathbb{C} \arrow[d,"0"']\arrow[r, shift left] \arrow[r, shift right] & \mathbb{C}^3 \arrow[d,"{\begin{psmallmatrix}0& 0 & 1\\0 & 0 & 0\end{psmallmatrix}}"']\arrow[r] & \mathbb{C} \arrow[d,"0"] \arrow[r, shift left] \arrow[r, shift right]  & \mathbb{C}^2 \arrow[d,"0"]\arrow[r, shift left] \arrow[r, shift right] & \mathbb{C}^2 \arrow[d,"0"] \\
&&&0\arrow[r, shift left] \arrow[r, shift right] & \mathbb{C}\arrow[r, shift left] \arrow[r, shift right] & \mathbb{C}^2 \arrow[r] & \mathbb{C}\arrow[r, shift left,"0"] \arrow[r, shift right,"1"']  & \mathbb{C} \arrow[r, shift left,"0"] \arrow[r, shift right,"1"'] & \mathbb{C} 
\end{tikzcd}
\]
\[
\begin{tikzcd}
0 \arrow[r, shift left] \arrow[r, shift right] & 0 \arrow[r, shift left] \arrow[r, shift right] & \mathbb{C}\arrow[r] &0\arrow[d] \arrow[r, shift left] \arrow[r, shift right] & \mathbb{C} \arrow[d,"0"']\arrow[r, shift left] \arrow[r, shift right] & \mathbb{C}^3 \arrow[d,"{\begin{psmallmatrix}0& 0 & 0\\0 & 0 & 1\end{psmallmatrix}}"']\arrow[r] & \mathbb{C} \arrow[d,"1"] \arrow[r, shift left] \arrow[r, shift right]  & \mathbb{C}^2 \arrow[d,"{\begin{psmallmatrix}0& 1\end{psmallmatrix}}"]\arrow[r, shift left] \arrow[r, shift right] & \mathbb{C}^2 \arrow[d,"{\begin{psmallmatrix}0&1\end{psmallmatrix}}"] \\
&&&0\arrow[r, shift left] \arrow[r, shift right] & \mathbb{C}\arrow[r, shift left] \arrow[r, shift right] & \mathbb{C}^2 \arrow[r] & \mathbb{C}\arrow[r, shift left,"0"] \arrow[r, shift right,"1"']  & \mathbb{C} \arrow[r, shift left,"0"] \arrow[r, shift right,"1"'] & \mathbb{C} 
\end{tikzcd}
\]
\end{lemma}
\begin{proof}
The morphism $P_{\tilde{C}}^\bullet \rightarrow I_A^\bullet$ specified by the following
\[
\begin{tikzcd}[column sep = 0.75cm]
0 \arrow[r, shift left] \arrow[r, shift right] & 0 \arrow[r, shift left] \arrow[r, shift right] & \mathbb{C}\arrow[r] &0\arrow[d] \arrow[r, shift left] \arrow[r, shift right] & \mathbb{C} \arrow[d,"0"']\arrow[r, shift left] \arrow[r, shift right] & \mathbb{C}^3 \arrow[d,"{\begin{psmallmatrix}0& 0 & 1\end{psmallmatrix}}"']\arrow[r] & \mathbb{C} \arrow[d,"{\begin{psmallmatrix}0\\0\\0\\1\end{psmallmatrix}}"] \arrow[r, shift left] \arrow[r, shift right]  & \mathbb{C}^2 \arrow[d,"{\begin{psmallmatrix}0& 0 \\ 0&1\end{psmallmatrix}}"]\arrow[r, shift left] \arrow[r, shift right] & \mathbb{C}^2 \arrow[d,"0"] \\
&&&\mathbb{C}^2\arrow[r, shift left] \arrow[r, shift right] & \mathbb{C}^2\arrow[r, shift left] \arrow[r, shift right] & \mathbb{C} \arrow[r] & \mathbb{C}^4\arrow[r, shift left,"0"] \arrow[r, shift right,"1"']  & \mathbb{C}^2 \arrow[r, shift left,"0"] \arrow[r, shift right,"1"'] & 0 \arrow[r] &  \mathbb{C} \arrow[r, shift left] \arrow[r, shift right] & 0 \arrow[r, shift left] \arrow[r, shift right] & 0
\end{tikzcd}
\]
is homotopy equivalent to the zero morphism. The result follows by simply verifying that the composition of the above maps with step 7 agrees with the maps of step 6 up to homotopy equivalence.
\end{proof}
\begin{lemma}[Step 9]
The composition $P_{\tilde{C}}^\bullet \rightarrow Cone(P_2\rightarrow \tilde{P}) \rightarrow P_2[1]$ is given by the following morphism of complexes.
\[
\begin{tikzcd}
0 \arrow[r, shift left] \arrow[r, shift right] & 0 \arrow[r, shift left] \arrow[r, shift right] & \mathbb{C}\arrow[r] &0\arrow[d] \arrow[r, shift left] \arrow[r, shift right] & \mathbb{C} \arrow[d,"0"']\arrow[r, shift left] \arrow[r, shift right] & \mathbb{C}^3 \arrow[d,"{\begin{psmallmatrix}0& 0 & 1\\0 & 0 & 0\end{psmallmatrix}}"']\arrow[r] & \mathbb{C}\arrow[r, shift left] \arrow[r, shift right]  & \mathbb{C}^2 \arrow[r, shift left] \arrow[r, shift right] & \mathbb{C}^2  \\
&&&0\arrow[r, shift left] \arrow[r, shift right] & \mathbb{C}\arrow[r, shift left] \arrow[r, shift right] & \mathbb{C}^2 \\[-10pt]
0 \arrow[r, shift left] \arrow[r, shift right] & 0 \arrow[r, shift left] \arrow[r, shift right] & \mathbb{C}\arrow[r] &0\arrow[d] \arrow[r, shift left] \arrow[r, shift right] & \mathbb{C} \arrow[d,"0"']\arrow[r, shift left] \arrow[r, shift right] & \mathbb{C}^3 \arrow[d,"{\begin{psmallmatrix}0& 0 & 0\\0 & 0 & 1\end{psmallmatrix}}"']\arrow[r] & \mathbb{C}  \arrow[r, shift left] \arrow[r, shift right]  & \mathbb{C}^2 \arrow[r, shift left] \arrow[r, shift right] & \mathbb{C}^2  \\
&&&0\arrow[r, shift left] \arrow[r, shift right] & \mathbb{C}\arrow[r, shift left] \arrow[r, shift right] & \mathbb{C}^2
\end{tikzcd}
\]
\end{lemma}
\begin{proof}
The result follows from composing the morphisms of step 8 with the natural map $Cone(P_2 \rightarrow \tilde{P}) \rightarrow P_2[1]$ as in Lemma~\ref{lem:projective}.
\end{proof}
\begin{corollary}[Step 10]
The compositions of steps 5 and 9 agree.
\end{corollary}
\subsubsection{Verification of compatibility for sequence~\eqref{eq:second}}\label{verification2}
In this section, we perform an analogous calculation as in the above section to prove part 2 in the assumption of Lemma~\ref{lem:calculation}.
\begin{lemma}\label{lem:calculation2}
For each element of a basis $g \colon I_2 \rightarrow I_1$, there exists a unique $g'$ making the following two diagrams commute.
\[
\begin{tikzcd}
P \arrow[r] \arrow[d]&I_2 \arrow[r,"i_2"]\arrow[d,"g"]& A \arrow[d,"g'"] & A \arrow[r,"p_2"] \arrow[d,"g'"]& P_2[1] \arrow[r]\arrow[d,"S^{-1}g{[1]}"] & \tilde{P}[1]\arrow[d] \\
P \arrow[r]& I_1 \arrow[r,"i_1"]&D[1] & D[1] \arrow[r,"p_1"] & P_1[1] \arrow[r] & \tilde{P}[1] 
\end{tikzcd}
\]
\end{lemma}
We first summarize the steps, illustrated in figure~\ref{fig:2}, describing the explicit computations involved in the homotopy category of complexes. It suffices to check the following for a basis of the vector space $Hom(I_2,I_1)$ described in Lemma~\ref{lem:serre}.
\begin{enumerate}[(a)]
\item
In step $1$, we compute the composition of the morphism $i_1\circ g$ with the injective resolution $D[1] \xrightarrow{\sim} I_{D[1]}^\bullet$ of Lemma~\ref{lem:cdresolution}.
\item
In steps $2$ and $3$, we compute the unique morphism $A \rightarrow I_{D[1]}^\bullet$ making the right square in the left map of triangles commute.
\item
In steps $4,5$, and $6$, we compute the composition of the projective resolution $P_A^\bullet \xrightarrow{\sim} A$ with the composition $S^{-1}g[1] \circ p_2$.
\item
In steps $7$ and $8$, we compute the composition of the projective resolution $P_A^\bullet \xrightarrow{\sim} A$ with the composition $p_1 \circ g'$.
\item
In step $9$, we verify that the compositions from parts (c) and (d) agree.
\end{enumerate}
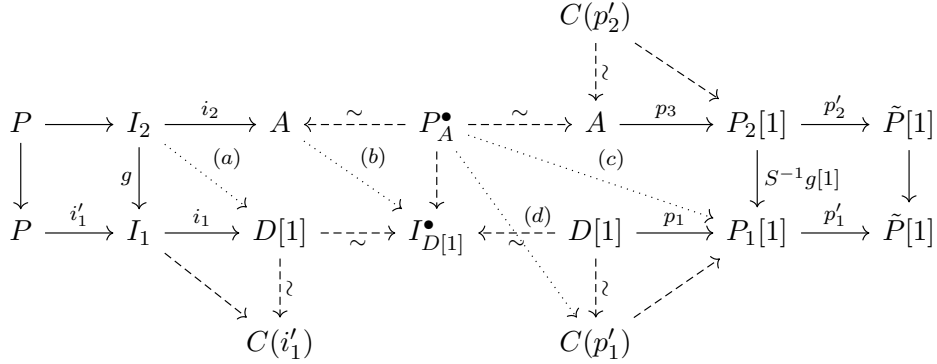
\begin{figure}[H]
\begin{tikzcd}
& & & &C(p_2') \isoarrow{d,dashed} \arrow[rd,dashed] & &\\
P \arrow[r] \arrow[d]&I_2 \arrow[r,"i_2"]\arrow[d,"g"']\arrow[rd,dotted,"(a)"]& A  \arrow[rd,dotted,"(b)"] &[-5pt]P_{A}^\bullet \arrow[rrd,dotted,"(c)"] \arrow[ rdd,dotted,"(d)"]\arrow[l,dashed,"\sim"'] \arrow[r,dashed,"\sim"] \arrow[d,dashed]&[-5pt]A \arrow[r,"p_3"]& P_2[1] \arrow[r,"p_2'"]\arrow[d,"S^{-1}g{[1]}"] & \tilde{P}[1]\arrow[d] \\
P \arrow[r,"i_1'"]& I_1 \arrow[r,"i_1"] \arrow[rd,dashed]&D[1]  \arrow[r,dashed,"\sim"'] \isoarrow{d,dashed}&  I_{D[1]}^\bullet & D[1] \isoarrow{d,dashed} \arrow[l,dashed,"\sim"] \arrow[r,"p_1"] & P_1[1] \arrow[r,"p_1'"] & \tilde{P}[1] \\
& & C(i_1')& &C(p_1') \arrow[ru,dashed] & & 
\end{tikzcd}
\caption{Illustration of the steps involved in the verification. In the diagram, $C(f)$ denotes the cone of a morphism $f$, the dashed arrows denote the canonical morphisms, the dashed isomorphisms come from Lemma~\ref{lem:aresolution},~\ref{lem:cdresolution}, and the dotted arrows correspond to the indicated steps.}
\label{fig:2}
\end{figure}
In more detail, the proof of Lemma~\ref{lem:calculation2} is broken into the following steps.
\begin{lemma}[Step 1]
The composition $I_2 \rightarrow I_1 \rightarrow D[1] \xrightarrow{\sim} I_{D[1]}^\bullet$ is given by the following morphism of complexes.
\[
\begin{tikzcd}
&&&\mathbb{C}^2 \arrow[d, "{\begin{psmallmatrix}0&0\\1&0\\0&0 \end{psmallmatrix}}"] \arrow[r, shift left] \arrow[r, shift right] &\mathbb{C} \arrow[d, "0"] \arrow[r, shift left] \arrow[r, shift right]  & 0\arrow[d] \\
\mathbb{C}^2 \arrow[r, shift left] \arrow[r, shift right]  & \mathbb{C}^2 \arrow[r, shift left] \arrow[r, shift right] & \mathbb{C} \arrow[r] & \mathbb{C}^3 \arrow[r, shift left] \arrow[r, shift right] & \mathbb{C} \arrow[r, shift left] \arrow[r, shift right] & 0 \arrow[r] & \mathbb{C} \arrow[r, shift left] \arrow[r, shift right] & 0 \arrow[r, shift left] \arrow[r, shift right] & 0
\end{tikzcd}
\]
\[
\begin{tikzcd}
&&&\mathbb{C}^2 \arrow[d, "{\begin{psmallmatrix}0&0\\0&1\\0&0 \end{psmallmatrix}}"] \arrow[r, shift left] \arrow[r, shift right] &\mathbb{C} \arrow[d, "0"] \arrow[r, shift left] \arrow[r, shift right]  & 0\arrow[d] \\
\mathbb{C}^2 \arrow[r, shift left] \arrow[r, shift right]  & \mathbb{C}^2 \arrow[r, shift left] \arrow[r, shift right] & \mathbb{C} \arrow[r] & \mathbb{C}^3 \arrow[r, shift left] \arrow[r, shift right] & \mathbb{C} \arrow[r, shift left] \arrow[r, shift right] & 0 \arrow[r] & \mathbb{C} \arrow[r, shift left] \arrow[r, shift right] & 0 \arrow[r, shift left] \arrow[r, shift right] & 0
\end{tikzcd}
\]
\end{lemma}
\begin{proof}
This follows from composing the basis of morphisms $g_1,g_2 \colon I_2 \rightarrow I_1$ of Lemma~\ref{lem:serre} with the morphism $I_1 \rightarrow I_{D[1]}^\bullet$ which is unique up to scalar multiplication.
\end{proof}
\begin{lemma}[Step 2]
The unique morphism $Cone(P\rightarrow I_2) \rightarrow I_{D[1]}^\bullet$ making the composition with $I_2 \rightarrow Cone(P\rightarrow I_2)$ agree with step 1 is given by the following morphism of complexes.
\[
\begin{tikzcd}
\mathbb{C}\arrow[r, shift left,"1"] \arrow[r, shift right,"0"']\arrow[d,"{\begin{psmallmatrix}-1\\0 \end{psmallmatrix}}"'] &\mathbb{C}\arrow[d,"{\begin{psmallmatrix}-1\\0 \end{psmallmatrix}}"'] \arrow[r, shift left,"1"] \arrow[r, shift right,"0"'] &\mathbb{C} \arrow[d,"-1"']\arrow[r]& \mathbb{C}^2 \arrow[d, "{\begin{psmallmatrix}0&0\\1&0\\0&0 \end{psmallmatrix}}"] \arrow[r, shift left] \arrow[r, shift right] &\mathbb{C} \arrow[d, "0"] \arrow[r, shift left] \arrow[r, shift right]  & 0\arrow[d] \\
\mathbb{C}^2 \arrow[r, shift left] \arrow[r, shift right]  & \mathbb{C}^2 \arrow[r, shift left] \arrow[r, shift right] & \mathbb{C} \arrow[r] & \mathbb{C}^3 \arrow[r, shift left] \arrow[r, shift right] & \mathbb{C} \arrow[r, shift left] \arrow[r, shift right] & 0 \arrow[r] & \mathbb{C} \arrow[r, shift left] \arrow[r, shift right] & 0 \arrow[r, shift left] \arrow[r, shift right] & 0
\end{tikzcd}
\]
\[
\begin{tikzcd}
\mathbb{C}\arrow[r, shift left,"1"] \arrow[r, shift right,"0"']\arrow[d,"0"] &\mathbb{C}\arrow[d,"0"] \arrow[r, shift left,"1"] \arrow[r, shift right,"0"'] &\mathbb{C} \arrow[d,"0"]\arrow[r]& \mathbb{C}^2 \arrow[d, "{\begin{psmallmatrix}0&0\\0&1\\0&0 \end{psmallmatrix}}"] \arrow[r, shift left] \arrow[r, shift right] &\mathbb{C} \arrow[d, "0"] \arrow[r, shift left] \arrow[r, shift right]  & 0\arrow[d] \\
\mathbb{C}^2 \arrow[r, shift left] \arrow[r, shift right]  & \mathbb{C}^2 \arrow[r, shift left] \arrow[r, shift right] & \mathbb{C} \arrow[r] & \mathbb{C}^3 \arrow[r, shift left] \arrow[r, shift right] & \mathbb{C} \arrow[r, shift left] \arrow[r, shift right] & 0 \arrow[r] & \mathbb{C} \arrow[r, shift left] \arrow[r, shift right] & 0 \arrow[r, shift left] \arrow[r, shift right] & 0
\end{tikzcd}
\]
\end{lemma}
\begin{proof}
This follows by simply composing the above morphism with the natural map $I_2 \rightarrow Cone(P \rightarrow I_2)$ and checking agreement with step 1.
\end{proof}
\begin{lemma}[Step 3]
The morphism $A\rightarrow I_{D[1]}^\bullet$ given by composing the natural map $Cone(P \rightarrow I_2) \xleftarrow{\sim} A$ with step 2 is given by the following morphism of complexes.
\[
\begin{tikzcd}
0\arrow[r, shift left] \arrow[r, shift right]\arrow[d] &\mathbb{C}\arrow[d,"{\begin{psmallmatrix}-1\\0 \end{psmallmatrix}}"'] \arrow[r, shift left,"1"] \arrow[r, shift right,"0"'] &\mathbb{C} \arrow[d,"-1"']\arrow[r]& \mathbb{C} \arrow[d, "0"] \arrow[r, shift left,"0"] \arrow[r, shift right,"1"'] &\mathbb{C} \arrow[d, "0"] \arrow[r, shift left] \arrow[r, shift right]  & 0\arrow[d] \\
\mathbb{C}^2 \arrow[r, shift left] \arrow[r, shift right]  & \mathbb{C}^2 \arrow[r, shift left] \arrow[r, shift right] & \mathbb{C} \arrow[r] & \mathbb{C}^3 \arrow[r, shift left] \arrow[r, shift right] & \mathbb{C} \arrow[r, shift left] \arrow[r, shift right] & 0 \arrow[r] & \mathbb{C} \arrow[r, shift left] \arrow[r, shift right] & 0 \arrow[r, shift left] \arrow[r, shift right] & 0
\end{tikzcd}
\]
\[
\begin{tikzcd}
0\arrow[r, shift left] \arrow[r, shift right]\arrow[d] &\mathbb{C}\arrow[d,"0"] \arrow[r, shift left,"1"] \arrow[r, shift right,"0"'] &\mathbb{C} \arrow[d,"0"']\arrow[r]& \mathbb{C} \arrow[d, "{\begin{psmallmatrix}0\\1\\0 \end{psmallmatrix}}"'] \arrow[r, shift left,"0"] \arrow[r, shift right,"1"'] &\mathbb{C} \arrow[d, "0"] \arrow[r, shift left] \arrow[r, shift right]  & 0\arrow[d] \\
\mathbb{C}^2 \arrow[r, shift left] \arrow[r, shift right]  & \mathbb{C}^2 \arrow[r, shift left] \arrow[r, shift right] & \mathbb{C} \arrow[r] & \mathbb{C}^3 \arrow[r, shift left] \arrow[r, shift right] & \mathbb{C} \arrow[r, shift left] \arrow[r, shift right] & 0 \arrow[r] & \mathbb{C} \arrow[r, shift left] \arrow[r, shift right] & 0 \arrow[r, shift left] \arrow[r, shift right] & 0
\end{tikzcd}
\]
\end{lemma}
\begin{proof}
This follows from a straightforward computation.
\end{proof}
\begin{lemma}[Step 4]
The composition $P_A^\bullet \xrightarrow{\sim} A \rightarrow I_{D[1]}^\bullet$ is given by the following morphism of complexes.
\[
\begin{tikzcd}[column sep = 0.75cm]
0\arrow[r, shift left] \arrow[r, shift right]&0\arrow[r, shift left] \arrow[r, shift right]&\mathbb{C} \arrow[r] &0\arrow[r, shift left] \arrow[r, shift right]\arrow[d] &\mathbb{C}^2\arrow[d,"{\begin{psmallmatrix}0&-1\\ 0 &0 \end{psmallmatrix}}"'] \arrow[r, shift left,"1"] \arrow[r, shift right,"0"'] &\mathbb{C}^4 \arrow[d,"{\begin{psmallmatrix}0\\0\\-1\\0 \end{psmallmatrix}}"]\arrow[r]& \mathbb{C} \arrow[d, "0"] \arrow[r, shift left] \arrow[r, shift right] &\mathbb{C}^2 \arrow[d] \arrow[r, shift left] \arrow[r, shift right]  & \mathbb{C}^2\arrow[d] \\
&&&\mathbb{C}^2 \arrow[r, shift left] \arrow[r, shift right]  & \mathbb{C}^2 \arrow[r, shift left] \arrow[r, shift right] & \mathbb{C} \arrow[r] & \mathbb{C}^3 \arrow[r, shift left] \arrow[r, shift right] & \mathbb{C} \arrow[r, shift left] \arrow[r, shift right] & 0 \arrow[r] & \mathbb{C} \arrow[r, shift left] \arrow[r, shift right] & 0 \arrow[r, shift left] \arrow[r, shift right] & 0
\end{tikzcd}
\]
\[
\begin{tikzcd}[column sep = 0.75cm]
0\arrow[r, shift left] \arrow[r, shift right]&0\arrow[r, shift left] \arrow[r, shift right]&\mathbb{C} \arrow[r] &0\arrow[r, shift left] \arrow[r, shift right]\arrow[d] &\mathbb{C}^2\arrow[d] \arrow[r, shift left,"1"] \arrow[r, shift right,"0"'] &\mathbb{C}^4 \arrow[d]\arrow[r]& \mathbb{C} \arrow[d, "{\begin{psmallmatrix}0\\1\\0 \end{psmallmatrix}}"'] \arrow[r, shift left] \arrow[r, shift right] &\mathbb{C}^2 \arrow[d] \arrow[r, shift left] \arrow[r, shift right]  & \mathbb{C}^2\arrow[d] \\
&&&\mathbb{C}^2 \arrow[r, shift left] \arrow[r, shift right]  & \mathbb{C}^2 \arrow[r, shift left] \arrow[r, shift right] & \mathbb{C} \arrow[r] & \mathbb{C}^3 \arrow[r, shift left] \arrow[r, shift right] & \mathbb{C} \arrow[r, shift left] \arrow[r, shift right] & 0 \arrow[r] & \mathbb{C} \arrow[r, shift left] \arrow[r, shift right] & 0 \arrow[r, shift left] \arrow[r, shift right] & 0
\end{tikzcd}
\]
\end{lemma}
\begin{proof}
This follows by composing the projective resolution of Lemma~\ref{lem:aresolution} with step 3.
\end{proof}
\begin{lemma}[Step 5]
The composition $P_A^\bullet \xrightarrow{\sim} A \xrightarrow{\sim} Cone(P_2 \rightarrow \tilde{P})$ is given by the following morphism of complexes.
\[
\begin{tikzcd}
0\arrow[r, shift left] \arrow[r, shift right]&0\arrow[r, shift left] \arrow[r, shift right]&\mathbb{C} \arrow[r] &0\arrow[r, shift left] \arrow[r, shift right]\arrow[d] &\mathbb{C}^2\arrow[d,"{\begin{psmallmatrix}0\\1 \end{psmallmatrix}}"'] \arrow[r, shift left,"1"] \arrow[r, shift right,"0"'] &\mathbb{C}^4 \arrow[d,"{\begin{psmallmatrix}0&0&1&0\\0&0&0&1 \end{psmallmatrix}}"]\arrow[r]& \mathbb{C} \arrow[d, "1"] \arrow[r, shift left] \arrow[r, shift right] &\mathbb{C}^2 \arrow[d,"{\begin{psmallmatrix}0&1 \end{psmallmatrix}}"] \arrow[r, shift left] \arrow[r, shift right]  & \mathbb{C}^2\arrow[d,"{\begin{psmallmatrix}0&1 \end{psmallmatrix}}"] \\
&&&0\arrow[r, shift left] \arrow[r, shift right]  & \mathbb{C} \arrow[r, shift left] \arrow[r, shift right] & \mathbb{C}^2 \arrow[r] & \mathbb{C}\arrow[r, shift left] \arrow[r, shift right] & \mathbb{C} \arrow[r, shift left] \arrow[r, shift right] & \mathbb{C}
\end{tikzcd}
\]
\end{lemma}
\begin{proof}
This follows by composing the projective resolution of Lemma~\ref{lem:aresolution} with the natural map $A \xleftarrow{\sim} Cone(P_2 \rightarrow \tilde{P})$ from Lemma~\ref{lem:projective}.
\end{proof}
\begin{lemma}[Step 6]
The composition $P_A^\bullet \xrightarrow{\sim}Cone(P_2 \rightarrow \tilde{P}) \rightarrow P_2[1] \rightarrow P_1[1]$ is given by the following morphism of complexes.
\[
\begin{tikzcd}
0\arrow[r, shift left] \arrow[r, shift right]&0\arrow[r, shift left] \arrow[r, shift right]&\mathbb{C} \arrow[r] &0\arrow[r, shift left] \arrow[r, shift right]\arrow[d] &\mathbb{C}^2\arrow[d,"{\begin{psmallmatrix}0&1\\0&0 \end{psmallmatrix}}"'] \arrow[r, shift left,"1"] \arrow[r, shift right,"0"'] &\mathbb{C}^4 \arrow[d,"{\begin{psmallmatrix}0&0&1&0\\0&0&0&0 \end{psmallmatrix}}"]\arrow[r]& \mathbb{C} \arrow[r, shift left] \arrow[r, shift right] &\mathbb{C}^2 \arrow[r, shift left] \arrow[r, shift right]  & \mathbb{C}^2\\
&&&\mathbb{C}\arrow[r, shift left] \arrow[r, shift right]  & \mathbb{C}^2 \arrow[r, shift left] \arrow[r, shift right] & \mathbb{C}^2
\end{tikzcd}
\]
\[
\begin{tikzcd}
0\arrow[r, shift left] \arrow[r, shift right]&0\arrow[r, shift left] \arrow[r, shift right]&\mathbb{C} \arrow[r] &0\arrow[r, shift left] \arrow[r, shift right]\arrow[d] &\mathbb{C}^2\arrow[d,"{\begin{psmallmatrix}0&0\\0&1 \end{psmallmatrix}}"'] \arrow[r, shift left,"1"] \arrow[r, shift right,"0"'] &\mathbb{C}^4 \arrow[d,"{\begin{psmallmatrix}0&0&0&0\\0&0&0&1 \end{psmallmatrix}}"]\arrow[r]& \mathbb{C} \arrow[r, shift left] \arrow[r, shift right] &\mathbb{C}^2 \arrow[r, shift left] \arrow[r, shift right]  & \mathbb{C}^2\\
&&&\mathbb{C}\arrow[r, shift left] \arrow[r, shift right]  & \mathbb{C}^2 \arrow[r, shift left] \arrow[r, shift right] & \mathbb{C}^2
\end{tikzcd}
\]
\end{lemma}
\begin{proof}
This follows by composing step 5 with the natural maps induced from Lemma~\ref{lem:serre}.
\end{proof}
\begin{lemma}[Step 7]
The unique morphism $P_A^\bullet \rightarrow D[1]$ whose composition with the injective resolution of Lemma~\ref{lem:cdresolution} agrees with step 4 is given by the following morphism of complexes.
\[
\begin{tikzcd}
0\arrow[r, shift left] \arrow[r, shift right]&0\arrow[r, shift left] \arrow[r, shift right]&\mathbb{C} \arrow[r] &0\arrow[r, shift left] \arrow[r, shift right]\arrow[d] &\mathbb{C}^2\arrow[d,"{\begin{psmallmatrix}0&-1\end{psmallmatrix}}"'] \arrow[r, shift left,"1"] \arrow[r, shift right,"0"'] &\mathbb{C}^4 \arrow[d,"{\begin{psmallmatrix}0\\0\\-1\\0 \end{psmallmatrix}}"]\arrow[r]& \mathbb{C} \arrow[r, shift left] \arrow[r, shift right] &\mathbb{C}^2 \arrow[r, shift left] \arrow[r, shift right]  & \mathbb{C}^2\\
&&&0\arrow[r, shift left] \arrow[r, shift right]  & \mathbb{C} \arrow[r, shift left,"1"] \arrow[r, shift right,"0"'] & \mathbb{C}
\end{tikzcd}
\]
\[
\begin{tikzcd}
0\arrow[r, shift left] \arrow[r, shift right]&0\arrow[r, shift left] \arrow[r, shift right]&\mathbb{C} \arrow[r] &0\arrow[r, shift left] \arrow[r, shift right]\arrow[d] &\mathbb{C}^2\arrow[d,"{\begin{psmallmatrix}1&0 \end{psmallmatrix}}"'] \arrow[r, shift left,"1"] \arrow[r, shift right,"0"'] &\mathbb{C}^4 \arrow[d,"{\begin{psmallmatrix}1\\0\\0\\0\end{psmallmatrix}}"]\arrow[r]& \mathbb{C} \arrow[r, shift left] \arrow[r, shift right] &\mathbb{C}^2 \arrow[r, shift left] \arrow[r, shift right]  & \mathbb{C}^2\\
&&&0\arrow[r, shift left] \arrow[r, shift right]  & \mathbb{C} \arrow[r, shift left,"1"] \arrow[r, shift right,"0"'] & \mathbb{C}
\end{tikzcd}
\]
\end{lemma}
\begin{proof}
The first morphism follows by directly checking the composition. For the second, we observe that the following morphism $P_A^\bullet \rightarrow I_{D[1]}^\bullet$ is homotopy equivalent to the zero morphism.
\[
\begin{tikzcd}[column sep = 0.75cm]
0\arrow[r, shift left] \arrow[r, shift right]&0\arrow[r, shift left] \arrow[r, shift right]&\mathbb{C} \arrow[r] &0\arrow[r, shift left] \arrow[r, shift right]\arrow[d] &\mathbb{C}^2\arrow[d,"{\begin{psmallmatrix}-1&0 \\0&0\end{psmallmatrix}}"'] \arrow[r, shift left,"1"] \arrow[r, shift right,"0"'] &\mathbb{C}^4 \arrow[d,"{\begin{psmallmatrix}-1&0&0&0 \end{psmallmatrix}}"']\arrow[r]& \mathbb{C} \arrow[d, "{\begin{psmallmatrix}0\\1\\0 \end{psmallmatrix}}"'] \arrow[r, shift left] \arrow[r, shift right] &\mathbb{C}^2 \arrow[d] \arrow[r, shift left] \arrow[r, shift right]  & \mathbb{C}^2\arrow[d] \\
&&&\mathbb{C}^2 \arrow[r, shift left] \arrow[r, shift right]  & \mathbb{C}^2 \arrow[r, shift left] \arrow[r, shift right] & \mathbb{C} \arrow[r] & \mathbb{C}^3 \arrow[r, shift left] \arrow[r, shift right] & \mathbb{C} \arrow[r, shift left] \arrow[r, shift right] & 0 \arrow[r] & \mathbb{C} \arrow[r, shift left] \arrow[r, shift right] & 0 \arrow[r, shift left] \arrow[r, shift right] & 0
\end{tikzcd}
\]
Subtracting the above from the second result of step 4, we see that the second morphism works.
\end{proof}
\begin{lemma}[Step 8]
The composition $P_A^\bullet \rightarrow D[1] \rightarrow P_1[1]$ obtained from composing step 7 with the natural map $D[1] \rightarrow P_1[1]$ in Lemma~\ref{lem:projective} is given by the following morphism of complexes.
\[
\begin{tikzcd}
0\arrow[r, shift left] \arrow[r, shift right]&0\arrow[r, shift left] \arrow[r, shift right]&\mathbb{C} \arrow[r] &0\arrow[r, shift left] \arrow[r, shift right]\arrow[d] &\mathbb{C}^2\arrow[d,"{\begin{psmallmatrix}0&1 \\0&0\end{psmallmatrix}}"'] \arrow[r, shift left,"1"] \arrow[r, shift right,"0"'] &\mathbb{C}^4 \arrow[d,"{\begin{psmallmatrix}0&0&1&0\\0&0&0&0 \end{psmallmatrix}}"]\arrow[r]& \mathbb{C} \arrow[r, shift left] \arrow[r, shift right] &\mathbb{C}^2 \arrow[r, shift left] \arrow[r, shift right]  & \mathbb{C}^2\\
&&&\mathbb{C}\arrow[r, shift left] \arrow[r, shift right]  & \mathbb{C}^2 \arrow[r, shift left,"1"] \arrow[r, shift right,"0"'] & \mathbb{C}^2
\end{tikzcd}
\]
\[
\begin{tikzcd}
0\arrow[r, shift left] \arrow[r, shift right]&0\arrow[r, shift left] \arrow[r, shift right]&\mathbb{C} \arrow[r] &0\arrow[r, shift left] \arrow[r, shift right]\arrow[d] &\mathbb{C}^2\arrow[d,"{\begin{psmallmatrix}0&0\\0&1 \end{psmallmatrix}}"'] \arrow[r, shift left,"1"] \arrow[r, shift right,"0"'] &\mathbb{C}^4 \arrow[d,"{\begin{psmallmatrix}0&0&0&0\\0&0&0&1 \end{psmallmatrix}}"]\arrow[r]& \mathbb{C} \arrow[r, shift left] \arrow[r, shift right] &\mathbb{C}^2 \arrow[r, shift left] \arrow[r, shift right]  & \mathbb{C}^2\\
&&&\mathbb{C}\arrow[r, shift left] \arrow[r, shift right]  & \mathbb{C}^2 \arrow[r, shift left] \arrow[r, shift right] & \mathbb{C}^2
\end{tikzcd}
\]
\end{lemma}
\begin{proof}
The first morphism follows from a direct calculation. For the second, we observe that the following morphism $P_A^\bullet \rightarrow P_1[1]$ is homotopy equivalent to the zero morphism.
\[
\begin{tikzcd}
0\arrow[r, shift left] \arrow[r, shift right]&0\arrow[r, shift left] \arrow[r, shift right]&\mathbb{C} \arrow[r] &0\arrow[r, shift left] \arrow[r, shift right]\arrow[d] &\mathbb{C}^2\arrow[d,"{\begin{psmallmatrix}1&0 \\0&1\end{psmallmatrix}}"'] \arrow[r, shift left,"1"] \arrow[r, shift right,"0"'] &\mathbb{C}^4 \arrow[d,"{\begin{psmallmatrix}1&0&0&0\\0&0&0&1\end{psmallmatrix}}"]\arrow[r]& \mathbb{C} \arrow[r, shift left] \arrow[r, shift right] &\mathbb{C}^2 \arrow[r, shift left] \arrow[r, shift right]  & \mathbb{C}^2\\
&&&\mathbb{C}\arrow[r, shift left] \arrow[r, shift right]  & \mathbb{C}^2 \arrow[r, shift left,"1"] \arrow[r, shift right,"0"'] & \mathbb{C}^2
\end{tikzcd}
\]
Composing the second result of step 7 with the morphism $D[1] \rightarrow P_1[1]$, we obtain the morphism
\[
\begin{tikzcd}
0\arrow[r, shift left] \arrow[r, shift right]&0\arrow[r, shift left] \arrow[r, shift right]&\mathbb{C} \arrow[r] &0\arrow[r, shift left] \arrow[r, shift right]\arrow[d] &\mathbb{C}^2\arrow[d,"{\begin{psmallmatrix}-1&0 \\0&0\end{psmallmatrix}}"'] \arrow[r, shift left,"1"] \arrow[r, shift right,"0"'] &\mathbb{C}^4 \arrow[d,"{\begin{psmallmatrix}-1&0&0&0\\0&0&0&0\end{psmallmatrix}}"]\arrow[r]& \mathbb{C} \arrow[r, shift left] \arrow[r, shift right] &\mathbb{C}^2 \arrow[r, shift left] \arrow[r, shift right]  & \mathbb{C}^2\\
&&&\mathbb{C}\arrow[r, shift left] \arrow[r, shift right]  & \mathbb{C}^2 \arrow[r, shift left,"1"] \arrow[r, shift right,"0"'] & \mathbb{C}^2
\end{tikzcd}
\]
Adding the above two morphisms, we obtain the second result.
\end{proof}
\begin{corollary}[Step 9]
The compositions of steps $6$ and $8$ agree.
\end{corollary}
\begin{proof}
This follows by directly comparing the results of steps 6 and 8.
\end{proof}
\section{Spherical functors associated to a categorical resolution}\label{sec:spherical}
In this section, we study a natural functor associated to any projective singular scheme. More precisely, we will be interested in the following situation.
\begin{replemma}{lem:resolutiondiagram}
Let $X_0 \xhookrightarrow{} \mathbb{P}^n$ be a projective scheme. Then there exists a diagram
\[
\begin{tikzcd}[row sep = 0.01cm]
& Perf(X_0) \arrow[ld, "\pi^*"']  &\\
\mathcal{D}^c \arrow[rd, "\pi_*"'] && D^b(\mathbb{P}^n)\arrow[lu, "i^*"']\\
& D^b(X_0)  \arrow[uu,symbol=\supseteq]\arrow[ru, "i_*"']&
\end{tikzcd}
\]
where $\mathcal{D}$ is a categorical resolution of $X_0$ equipped with a pair of functors $(\pi^*, \pi_*)$  and $i_*$ is induced from the closed immersion $i \colon X \xhookrightarrow{} \mathbb{P}^n$ with left adjoint $i^*$.
\end{replemma}
The goal of this section is to prove the following Proposition.
\begin{repprop}{prop:bondalsphericalfunctor}
Assume $\mathcal{D} = P^\perp \subset D^b(Q)$ with $Q$ the Bondal quiver and $X_0$ the nodal cubic in $\mathbb{P}^2$. Then the composition $F \coloneqq i_* \circ \pi_*$ is a spherical functor.
\end{repprop}
This section is organized as follows. In subSection~\ref{catresolution}, we review the notion of a categorical resolution, prove Lemma~\ref{lem:resolutiondiagram}, and record some general facts about spanning classes in our setting. In subSection~\ref{spherical}, we review the notion of a spherical functor and set up the main approach to Proposition~\ref{prop:bondalsphericalfunctor}. In subSection~\ref{bondalresolution}, we prove Proposition~\ref{prop:bondalsphericalfunctor}. Aside from subSection~\ref{bondalresolution}, we will always assume $X_0 \xhookrightarrow{} \mathbb{P}^n$ an integral projective scheme.
\subsection{Categorical Resolutions}\label{catresolution}
We first review the basic notion of a categorical resolution.
\begin{defn}[\cite{10.1093/imrn/rnu072}, Definition 1.3]\label{def:catresolution}
A {\em categorical resolution} of a scheme $Y$ is a smooth, cocomplete, compactly generated triangulated category $\mathcal{D}$ together with an adjoint pair of triangulated functors
\[
\pi^* \colon D(Y) \rightarrow \mathcal{D}, \quad  \pi_* \colon \mathcal{D} \rightarrow D(Y)
\]
satisfying the following
\begin{enumerate}
\item
$\pi_* \circ \pi^* = id$
\item
$\pi^*$ and $\pi_*$ commute with arbitrary direct sums
\item
$\pi_*(\mathcal{D}^c) \subset D^b(coh(Y))$
\end{enumerate}
\end{defn}
The utility of this definition is the following theorem.
\begin{theorem}\label{thm:catresolution}[\cite{10.1093/imrn/rnu072}, Theorem 1.4]
Any separated scheme $Y$ of finite type over a field of characteristic $0$ admits a categorical resolution.
\end{theorem}
By using the above theorem, we state the main setting of interest.
\begin{lemma}\label{lem:resolutiondiagram}
There exists a diagram
\[
\begin{tikzcd}[row sep = 0.01cm]
& Perf(X_0) \arrow[ld, "\pi^*"']  &\\
\mathcal{D}^c \arrow[rd, "\pi_*"'] && D^b(\mathbb{P}^n)\arrow[lu, "i^*"']\\
& D^b(X_0)  \arrow[uu,symbol=\supseteq]\arrow[ru, "i_*"']&
\end{tikzcd}
\]
where $\mathcal{D}$ is a categorical resolution of $X_0$ equipped with a pair of functors $(\pi^*, \pi_*)$  and $i_*$ is induced from the closed immersion $i \colon X_0 \xhookrightarrow{} \mathbb{P}^n$ with left adjoint $i^*$.
\end{lemma}
\begin{proof}
For the left pair of functors, we simply apply theorem~\ref{thm:catresolution}. Restricting to the subcategory of compact objects $\mathcal{D}^c \subseteq \mathcal{D}$ and $D^b(X_0) \subseteq D(X_0)$ respectively, we have that $\pi_*(\mathcal{D}^c) \subset D^b(X_0)$ by definition~\ref{def:catresolution}(3). On the other hand, definition~\ref{def:catresolution}(2) implies that $\pi^*(Perf(X_0)) \subset \mathcal{D}^c$ by \cite[Lemma 2.10]{10.1093/imrn/rnu072}.

For the right pair of functors, the adjunction $(i^*, i_*) \colon D^b(X_0) \leftrightarrows D^b(\mathbb{P}^n)$ is standard. In particular, $i^*(D^b(\mathbb{P}^n)) \subseteq Perf(X_0)$ as any object $\mathcal{F}^\bullet \in D^b(\mathbb{P}^n)$ admits a finite, locally free resolution and the pullbacks of locally free sheaves are locally free. Thus, the claim follows.
\end{proof}

In general, it is desirable to work with categories and functors of a geometric nature. We recall the following definition.
\begin{defn}
A $k$-linear triangulated category $\mathcal{D}$ is a {\em geometric noncommutative scheme} if there exists a fully faithful functor $F\colon\mathcal{D} \xhookrightarrow{} D^b(X)$ admitting left and right adjoints $F_L, F_R$ with $X$ a smooth, projective variety.
\end{defn}
The theory of Fourier-Mukai transforms admits a natural definition in the setting of geometric noncommutative schemes as well.
\begin{defn}
Given geometric noncommutative schemes $\mathcal{A}, \mathcal{B}$ with embeddings 
\[
G \colon \mathcal{A} \xhookrightarrow{} D^b(X),\quad H \colon \mathcal{B} \xhookrightarrow{} D^b(Y)
\]
a functor $F \colon \mathcal{A} \rightarrow \mathcal{B}$ is of Fourier-Mukai type if the composition
\[
D^b(X) \xrightarrow{G_L} \mathcal{A} \xrightarrow{F} \mathcal{B} \xrightarrow{H} D^b(Y)
\]
is of Fourier-Mukai type.
\end{defn}
Similarly, given a geometric noncommutative scheme $G \colon \mathcal{D} \rightarrow D^b(X)$ and a possibly singular projective scheme $X_0$, we will say that a functor $F \colon D^b(X_0) \rightarrow \mathcal{D}$ is Fourier-Mukai, if the composition
\[
D^b(X_0) \xrightarrow{F} \mathcal{D} \xrightarrow{G} D^b(X)
\]
is representable by a kernel $E \in D^b(X_0 \times X)$, and similarly for functors $F' \colon \mathcal{D} \rightarrow D^b(X_0)$.

\begin{lemma}\label{lem:fouriermukai}
In the setting of Lemma~\ref{lem:resolutiondiagram}, assume in addition that $\mathcal{D}^c$ is a geometric noncommutative scheme. Then the functors $(\pi^*, \pi_*)$ and $(i^*, i_*)$ are of Fourier-Mukai type. 
\end{lemma}
\begin{proof}
We begin again with the pair of functors $(\pi^*,\pi_*)$. By assumption, there exists a fully faithful functor with left and right adjoints $F \colon \mathcal{D}^c \xhookrightarrow{} D^b(X)$ where $X$ is a smooth, projective variety. Taking the composition with the categorical resolution $\pi^*$, there exists a fully faithful functor $F \circ \pi^* \colon Perf(X_0) \xhookrightarrow{} D^b(X)$. By \cite[Corollary 9.13]{2010JAMS...23..853L}, the composition $F \circ \pi^*$ is of Fourier-Mukai type and so by definition, the resolution functor $\pi^*$ is Fourier-Mukai. 

For $\pi_*$, it suffices to prove that the composition $\pi_* \circ F_R \colon D^b(X) \rightarrow Perf(X_0)$ is of Fourier-Mukai type. Noting that $F \circ \pi^*$ satisfies the assumption of \cite[Lemma 3.17]{ballard2009equivalences} and combining the result with \cite[Lemma 3.21]{ballard2009equivalences}, it follows that this composition is indeed Fourier-Mukai.

For the pair of functors $(i^*, i_*)$, the standard proof that these are representable by the kernel $\mathcal{O}_{\Gamma_i} \in D^b(X_0 \times \mathbb{P}^n)$ implies the result, where $\Gamma_i \subset X_0 \times \mathbb{P}^n$ is the graph of the inclusion $i \colon X_0 \xhookrightarrow{} \mathbb{P}^n$. 
\end{proof}
For later use in Section~\ref{bondalresolution}, we will record the following facts about spanning classes in the generality of the present setting.
\begin{lemma}[\cite{2010JAMS...23..853L}, Proposition 9.2]\label{lem:spanning}
The subset of objects $\{\mathcal{O}_{X_0}(i) = i^* \mathcal{O}_{\mathbb{P}^n}(i) \,|\, i \in \mathbb{Z} \}$ is an ample sequence in $Coh(X_0)$.
\end{lemma}
\begin{lemma}\label{lem:orthogonalample}
Let $D \in \mathcal{D}^c$ be any object. Then the following are equivalent.
\begin{enumerate}
\item
$Hom(\pi^*\mathcal{O}_{X_0}(i),D) = 0$ for all $i \in \mathbb{Z}$.
\item
$\pi_*D = 0$.
\item
$D \in Im(\pi^*)^\perp$.
\end{enumerate}
\end{lemma}
\begin{proof}
$(1) \implies (2)\colon$ Assume that $Hom(\pi^*\mathcal{O}_{X_0}(i), D) = 0$ for all $i \in \mathbb{Z}$. Then by adjunction, $Hom(\mathcal{O}_{X_0}(i), \pi_* D) = 0$ in $D^b(X_0)$ for all $ i \in \mathbb{Z}$. By Lemma~\ref{lem:spanning}, the set $\{\mathcal{O}_{X_0}(i)\}_{i \in \mathbb{Z}}$ is an ample sequence in $Coh(X_0)$. By part $1$ of the proof of \cite[Proposition 2.73]{huybrechts2006fourier}, we see that $Hom(\mathcal{O}_{X_0}(i), \pi_*D) = 0 \implies \pi_*D = 0$ without assumptions on the homological dimension of $Coh(X_0)$.

$(2) \implies (3) \colon$This follows immediately by adjunction.

$(3) \implies (1) \colon$This is immediate.
\end{proof}
\subsection{Spherical Functors}\label{spherical}
As a result of Lemma~\ref{lem:fouriermukai}, it makes sense to talk about spherical functors in our context. In this subsection, we will always assume $F \colon \mathcal{A} \rightarrow \mathcal{B}$ an exact functor between triangulated categories, with left and right adjoints $L, R \colon \mathcal{B} \rightarrow \mathcal{A}$.
\begin{defn}\cite{Addington2016NewDS}
\begin{enumerate}
\item
The {\em twist} $T_F$ associated to $F$ is the cone of the counit of adjunction
\[
\begin{tikzcd}
FR \arrow[r] & id_\mathcal{B} \arrow[r] & T_F
\end{tikzcd}
\]
\item
The {\em cotwist} $C_F$ associated to $F$ is the cone of the unit of adjunction
\[
\begin{tikzcd}
id_\mathcal{A} \arrow[r] & RF \arrow[r] & C_F
\end{tikzcd}
\]
\item
$F$ is a {\em spherical} if both $T_F$ and $C_F$ are equivalences of categories.
\end{enumerate}
\end{defn}
We now restate the criterion of \cite[Theorem 1]{Addington2016NewDS} in a form directly applicable to our setting.
\begin{theorem}[\cite{Addington2016NewDS}, Theorem 1]\label{lem:spherical}
Assume the following:
\begin{enumerate}
\item
The cotwist $C_F$ is an equivalence.
\item
There exists an equivalence $S_B F C \simeq FS_A$.
\end{enumerate}
Then $F$ is a spherical functor. 
\end{theorem}

\subsection{Bondal Quiver}\label{bondalresolution}
We now finally specialize to the case of the Bondal quiver $Q$. We take $X_0 \subset \mathbb{P}^2$ the nodal cubic. The main goal of this section is to prove the following Proposition.
\begin{prop}\label{prop:bondalsphericalfunctor}
Assume $\mathcal{D} = P^\perp \subset D^b(Q)$ with $Q$ the Bondal quiver and $X_0$ the nodal cubic in $\mathbb{P}^2$. Then the composition $F \coloneqq i_* \circ \pi_*$ is a spherical functor.
\end{prop}
Though one may work with the prescription of \cite{10.1093/imrn/rnu072} to construct the categorical resolution explicitly, we will instead use the established results of \cite{2009arXiv0905.1231B} for convenience.
\begin{theorem}[\cite{2009arXiv0905.1231B}]\label{thm:quiverresolution}
There exists a pair of adjoint functors $(\pi^*, \pi_*) \colon D^b(Q) \leftrightarrows Perf(X_0), D^b(X_0)$, satisfying the following:
\begin{enumerate}[(a)]
\item
The pair $(\pi^*, \pi_*)$ together with $D^b(Q)$ is a categorical resolution of $X_0$ in the sense of definition~\ref{def:catresolution}.
\item
The functor $\pi^*$ maps:
\begin{enumerate}[(1)]
\item
torsion sheaves of regular points on $X_0$ to the family of quivers
\[
\begin{tikzcd}
\mathbb{C} \arrow[r, shift left, "a"]\arrow[r, shift right, "b"'] & \mathbb{C} \arrow[r, shift left] \arrow[r, shift right] & 0
\end{tikzcd}
\]
with $a, b \neq 0$.
\item
the Jacobian $Pic^0(E)$ to the family of complexes
\[
\begin{tikzcd}
P_3 \arrow[r,"{\begin{psmallmatrix} k\alpha_1\\ \beta_1\end{psmallmatrix}}"] & P_2 \oplus P_2 \arrow[r,"{\begin{psmallmatrix}\alpha_2 & \beta_2\end{psmallmatrix}}"] & P_1
\end{tikzcd}
\]
where $\alpha_i,\beta_j$ are the canonical morphisms on the projective modules induced by the paths in the quiver and $k \in \mathbb{C}^*$.
\end{enumerate}
\end{enumerate}
\end{theorem}
\begin{proof}
This is essentially a rephrasing of Theorem 2.6, Proposition 7.2, and Proposition 7.4 of \cite{2009arXiv0905.1231B}.
\end{proof}
In the rest of this subsection, we fix the pair of adjoint functors $(\pi^*, \pi_*) \colon  D^b(Q) \leftrightarrows Perf(X_0), D^b(X_0)$ as in theorem~\ref{thm:quiverresolution}. The identification of the sheaves in theorem~\ref{thm:quiverresolution}(b) with the corresponding objects in $D^b(Q)$ allows us to easily deduce the following.
{\setlength{\emergencystretch}{0.5\textwidth}\par}
\begin{lemma}\label{lem:KerR}
We have the inclusion $P, \tilde{P} \in Ker(\pi_*)$.
\end{lemma}
\begin{proof}
We first consider the object $P \in D^b(Q)$. By Lemma~\ref{lem:orthogonalample}, it suffices to prove that $Hom(\pi^*\mathcal{O}_{X_0}(i), P) = 0$ for all $i \in \mathbb{Z}$. A straightforward calculation with the family of quivers of theorem~\ref{thm:quiverresolution}(b) implies in particular, that $S_Q(\pi^* \mathcal{O}_{X_0}) \simeq \pi^* \mathcal{O}_{X_0}[1]$ and $S_Q(\pi^* k(p)) \simeq \pi^*k(p) [1]$ for $p \in X_0$ any regular point. This implies, by boundedness of cohomology, that $Hom(\pi^*\mathcal{O}_{X_0}, P) = Hom(\pi^*k(p), P) = 0$ for $p \in X_0$ any regular point as $P$ is $(4,2)$-Calabi-Yau by Lemma~\ref{lem:p}. To deduce the claim, observe that we have the exact triangles
{\setlength{\emergencystretch}{0.5\textwidth}\par}
\[
\begin{tikzcd}[row sep = 0.5cm]
\mathcal{O}_{X_0}(-1) \arrow[r]& \mathcal{O}_{X_0} \arrow[r]& k(p_1) \oplus k(p_2) \oplus k(p_3)\\
(k(p_1) \oplus k(p_2) \oplus k(p_3)) \arrow[r] & \mathcal{O}_{X_0}[1] \arrow[r] & \mathcal{O}_{X_0}(1)[1] \\
\end{tikzcd}
\]
where $p_1,p_2,p_3$ are three regular points on a line through $X_0 \subset \mathbb{P}^2$. Proceeding by induction, it follows that $Hom(\pi^*\mathcal{O}_{X_0}(i), P) = 0$ for all $i \in \mathbb{Z}$. The argument is identical for the case of $\tilde{P}$ and we conclude.
\end{proof}
\begin{lemma}\label{lem:Eintersection}
Let $E \in P^\perp$ be the unique $3$-spherical object. Consider the image $Im(\pi^*)$ of the categorical resolution $Perf(X_0) \xrightarrow{\pi^*} P^\perp$. Define $Im(\pi^*)_E^\perp = \prescript{\perp}{E}{Im(\pi^*)}$ to be the right/left orthogonal subcategory within $E^\perp$. Then the following holds.
\begin{enumerate}
\item
$E^\perp = P^\perp \cap \prescript{\perp}{}{P} = \{T \in D^b(Q) \,|\, S_Q(T) = T[1]\}$.
\item
$Im(\pi^*) \subset E^\perp$.
\end{enumerate}
\end{lemma}
\begin{proof}
For the first part, note that if $A \in P^\perp$ satisfied $Hom(E, A[i]) = 0$ for all $ i \in \mathbb{Z}$, then $A \in E^\perp$. But if $Hom(B, A[i]) = 0$ for all $ i \in \mathbb{Z}$ and $B \in E^\perp$, then in particular $Hom(A,A) = 0$ and so $A \simeq 0$. The other direction is identical by using that $E^\perp = \prescript{\perp}{}{E}$ by Serre duality as $E$ is Calabi-Yau.

For the second, it suffices to prove by part~$(1)$ that for any $A \in Perf(X_0)$, we have the vanishings $Hom(\pi^*A, P[i]) = 0, Hom(P, \pi^*A[j]) = 0$ for all $i,j \in \mathbb{Z}$. But this follows from Lemma~\ref{lem:KerR} together with Serre duality and adjunction.

%
\end{proof}


We now turn to proving Proposition~\ref{prop:bondalsphericalfunctor}.
\begin{proof}[Proof of Proposition~\ref{prop:bondalsphericalfunctor}]
We have the following diagram
\[
\begin{tikzcd}[row sep = 0.3cm]
D^b(\mathbb{P}^2) \arrow[r,"i^*"]&Perf(X_0) \arrow[d,symbol=\subseteq] \arrow[r, "\pi^*"] & P^\perp\arrow[ll, bend right=30,"L"]\arrow[ld, "\pi_*"]\\
& D^b(X_0) \arrow[lu, "i_*"]&
\end{tikzcd}
\]
where we define $F = \pi^* \circ i^*$, the right adjoint $R = i_* \circ \pi_*$. As $D^b(\mathbb{P}^2)$ and $P^\perp$ admit Serre functors, the left adjoint exists and is given by $L = S_{\mathbb{P}^2}^{-1} \circ R \circ S_{P^\perp}$.

We first prove that the cotwist is an equivalence. By definition, $C_F$ fits into an exact triangle
\[
\begin{tikzcd}
id_{\mathbb{P}^2} \arrow[r] & i_* \pi_* \pi^* i^* \arrow[r] & C_F
\end{tikzcd}
\]
By definition~\ref{def:catresolution}(1), we have that $C_F = Cone(id_{\mathbb{P}^2} \rightarrow i_* i^*) = \otimes\mathcal{O}_{\mathbb{P}^2}(-3)[1]$ and the claim follows.

By theorem~\ref{lem:spherical}, it suffices to prove the equivalence $S_{P^\perp}FC \simeq F S_{D^b(\mathbb{P}^2)}$. On the other hand, by Lemma~\ref{lem:Eintersection} and the above paragraph, we have the equivalences 
\[S_{P^\perp}FC \simeq [1] \circ F C \simeq FC[1] \simeq F\circ (\otimes \mathcal{O}_{\mathbb{P}^2}(-3))[2] = FS_{D^b(\mathbb{P}^2)}
\]
and we immediately conclude.


\end{proof}
\begin{rem}
Though it is not central to our analysis, we remark that the resolution functor $\pi^*$ is indeed crepant in the sense of \cite[Definition 3.4]{2016}. To see this, let $A\in P^\perp, B \in Perf(X_0)$. Then the right adjoint to the functor $\pi_*$ is specified by the following sequence of isomorphisms
\[
Hom(\pi_*A,B) = Hom(S_{X_0}^{-1}B, \pi_*A) = Hom(\pi^* S_{X_0}^{-1}B, A) = Hom(A, S_{P^\perp}\pi^* S_{X_0}^{-1}B)
\]
In \cite{10.1093/imrn/rnu072}, the image $\pi^*B$ is defined by a pair $(i_0^*B, \nu^*B) \in D(pt) \times_\varphi D(\mathbb{P}^1)$ where $i_0 \colon Spec(k) \xhookrightarrow{} X_0$ is the inclusion of the node and $\nu \colon \mathbb{P}^1 \rightarrow X_0$ is the normalization and $\varphi$ is a gluing bimodule. But the dualizing sheaf on $X_0$ satisfies $\nu^*\omega_{X_0} \simeq \mathcal{O}_{\mathbb{P}^1}$ and so we have
\[
S_{P^\perp}\pi^*S_{X_0}^{-1}B \simeq S_{P^\perp}\pi^* (B \otimes \omega_{X_0}^{-1} [-1] )\simeq S_{P^\perp} \pi^*B [-1] \simeq \pi^*B 
\]
Thus $\pi^*$ is also right adjoint to $\pi_*$ and so the resolution functor is crepant. In particular, both $P^\perp$ and $D^b(Q)$ serve as crepant resolutions in this case, and hence this gives an example of a non-unique crepant categorical resolution.
\end{rem}

\bibliographystyle{amsplain}
\bibliography{bondalserre}
\end{document}